\documentclass[a4paper, reqno]{amsart}
\usepackage{graphicx}
\usepackage[bookmarksnumbered,colorlinks,plainpages]{hyperref}

\usepackage{float}
\usepackage{amsthm}
\usepackage[all,2cell]{xy}
\usepackage{amsfonts}
\usepackage{amsmath}
\usepackage{amssymb}
\usepackage{xypic,leqno,amscd,amssymb,pstricks,latexsym,amsbsy,xypic,mathrsfs,verbatim}
\usepackage{multirow}
\usepackage{booktabs} 
\usepackage{makecell} 

\usepackage{etex}
\usepackage{mathtools}

\usepackage{tikz}
\usetikzlibrary{positioning,arrows}
\usetikzlibrary{decorations.pathreplacing}

\UseAllTwocells \SilentMatrices
\usepackage{amscd,amssymb,pstricks,latexsym,amsbsy,xypic,mathrsfs,verbatim}
\usepackage[all]{xy}
\usepackage{lscape}
\usepackage{enumerate}
\usepackage{graphicx}
\usepackage{bm}
\usepackage{amsmath}
\usepackage{xypic,leqno,amscd,amssymb,pstricks,latexsym,amsbsy,xypic,mathrsfs,verbatim}
\usepackage{multirow}
\usepackage{booktabs} 
\usepackage{makecell}
\usepackage{mathrsfs}

\makeatletter
\@namedef{subjclassname@2010}{%
  \textup{2020} Mathematics Subject Classification}

\numberwithin{equation}{section}

\theoremstyle{plain}
\newtheorem{prop}{Proposition}[section]
\newtheorem{thm}[prop]{Theorem}
\newtheorem{cor}[prop]{Corollary}
\newtheorem{lem}[prop]{Lemma}
\newtheorem{defn}[prop]{Definition}

\theoremstyle{definition}
\newtheorem{exam}[prop]{Example}
\newtheorem{rem}[prop]{Remark}

\def\End{{{\rm End}\,}}
\def\Hom{{{\rm Hom}\,}}

\def\ind{{{\rm ind}\,}}
\def\add{{{\rm add}\,}}
\def\im{{{\rm im}\,}}
\def\soc{{{\rm soc}\,}}
\def\top{{{\rm top}\,}}
\def\Ext{{{\rm Ext}\,}}

\def\D{{{\rm D}\,}}

\def\DHom{{{\rm DHom}\,}}
\def\mod{{\text{\rm mod}}}

\def\cal{\mathcal}

\def\bbZ{{\mathbb Z}}

\def\bbL{{\mathbb L}}

\def\co{{\mathcal O}}
\def\vx{{\vec{x}}}

\def\vc{{\vec{c}}}

\def\coh{{\rm coh\,}}
\def\vect{{\rm vect\,}}
\def\X{{\mathbb X}}

\def\add{{\rm add}}
\def\deg{{\rm deg}}
\def\rank{{\rm rank}}

 \def\lz{\lambda}

\def\End{{{\rm End}\,}} 
\def\Hom{{{\rm Hom}\,}}

\def\soc{{{\rm soc}\,}} 
\def\top{{{\rm top}\,}}
\def\Ext{{{\rm Ext}\,}}

\def\D{{{\rm D}\,}}

\def\DHom{{{\rm DHom}\,}}
\def\mod{{\text{\rm mod}\,}}

\def\cal{\mathcal}

\def\bbZ{{\mathbb Z}}

\def\bbL{{\mathbb L}}

\def\add{{\rm{add}}}

\def\co{{\mathcal O}}
\def\vx{{\vec{x}}}

\def\vc{{\vec{c}}}

\def\coh{{\rm coh\,}}
\def\vect{{\rm vect\,}}
\def\X{{\mathbb X}}

\begin{document}

\title[{Stability approach to torsion pairs on abelian categories}]
{Stability approach to torsion pairs on abelian categories}

\author[Mingfa Chen, Yanan Lin, Shiquan Ruan] {Mingfa Chen, Yanan Lin, Shiquan Ruan$^*$}

\address{Mingfa Chen, Yanan Lin and Shiquan Ruan;\; School of Mathematical Sciences, Xiamen University, Xiamen, 361005, Fujian, P.R. China.}
\email{mingfachen@stu.xmu.edu.cn, ynlin@xmu.edu.cn, sqruan@xmu.edu.cn}

\thanks{$^*$ the corresponding author}


\date{\today}
\makeatletter \@namedef{subjclassname@2020}{\textup{2020} Mathematics Subject Classification} \makeatother

\begin{abstract}
In this paper we introduce a local-refinement procedure to investigate stability data on an abelian category, and provide a sufficient and necessary condition for a stability data to be finest.
We classify all the finest stability data for the categories of coherent sheaves over certain weighted projective curves, including the classical projective line, smooth elliptic curves and certain weighted projective lines. As applications, we obtain a classification of torsion pairs for these categories via stability data approach. As a by-product, a new proof for the classification of torsion pairs in any tube category is also provided.
\end{abstract}

\maketitle
\section{Introduction}

Torsion pairs in abelian categories were introduced by Dickson \cite{Dickson}, which become increasingly important in a wide range of  research areas.
Mizuno and Thomas \cite{MT} found a close relationship between
$c$-sortable elements of a Coxeter group and torsion pairs, in terms of the representation theory of preprojective algebras, and described the cofinite torsion classes in the context of the Coxeter group.
Tattar \cite{Tat} defined torsion pairs for quasi-abelian categories and gave several characterizations, he
showed that many of the torsion theoretic concepts translate from abelian categories to quasi-abelian categories.

The wall and chamber structure of a module category was introduced by Bridgeland
in \cite{TB2} to give an algebraic interpretation of scattering diagrams studied in mirror symmetry by Gross, Hacking, Keel and Kontsevich, see \cite{GHKK}.
It is shown in \cite{BST1} that all functorially finite torsion classes of an algebra can be obtained from its wall and chamber structure. Asai \cite{Asai} proved that chambers coincide with the so-called TF equivalence classes (defined via numerical torsion pairs) by using Koenig-Yang correspondences in silting theory. Research highlights about torsion pairs are also studied in \cite{Ban,BL,BP,Enom,Ja,PaSV}.

The notion of stability data arising from Geometric Invariant Theory was introduced by Mumford  \cite{Mum} in order to construct the moduli spaces of vector bundles on an algebraic curve.
This new approach was adapted by a great deal of mathematicians
to different branches of mathematics.
Such is the case of Schofield,
who did it for quiver representations in \cite{Schofi}; King, for representation of finite
dimensional algebras in \cite{King}; Rudakov, for abelian categories in \cite{Rudak}; and
Bridgeland, for triangulated categories in \cite{TB1}.
The stability data is generalized to t-stability in \cite{GKR}, which provides an effective approach to classify the bounded t-structures on the derived categories of coherent sheaves
on the classical projective line and on an smooth elliptic curves.

The stability data on an abelian category has a close relation with torsion theory. Namely, for an abelian category $\mathcal{A}$,
any torsion pair gives a stability data; and for any stability data on $\mathcal{A}$, there exists a family of torsion pairs, c.f. \cite{BST2}. More precisely, any stability data induces a
chain of torsion classes in $\mathcal{A}$.
A finite non-refinable increasing chain of torsion classes starting with the zero class and ending in $\mathcal{A}$ is called a maximal green sequence in $\mathcal{A}$.
Br$\rm{\ddot{u}}$stle-Smith-Treffinger \cite{BST2} characterized which stability data induce maximal green sequences in $\mathcal{A}$.

The stability data have played a key role in the calculations of Donaldson-Thomas invariants, which have deep implications in algebraic geometry and theoretical physics.
As is described in \cite{TB3}, generating functions for Donaldson-Thomas invariants can be deduced from certain factorisation of a distinguished element $e_\mathcal{A}$ in the completed Hall algebra associated to $\mathcal{A}$ via an integration map.
The equalities between different factorisations of $e_\mathcal{A}$ induced by distinct stability data are known under the name of wall-crossing formulas.

In the present paper we study the finest stability data on abelian categories and apply them to the classification of torsion pairs.
We define a partial ordering for the set of stability data on an abelian category, minimal elements with respect to this partial ordering will be called the finest stability data.
We give a sufficient and necessary condition to determine when a
stability data is finest.
We classify all the finest stability data for the categories of coherent sheaves over certain weighted projective curves, including the classical projective line, smooth elliptic curves and the weighted projective line of weight type (2).
As an application, we obtain the classification of torsion pairs for these categories.
Moreover, for a nilpotent representation category of a cyclic quiver (i.e., a tube category), Baur-Buan-Marsh \cite{BBM} classified the torsion pairs by using maximal rigid objects in the extension of the tube category. We provide a new proof via stability data approach.

This paper is organized as follows.
In Section 2, we recall the definition and some basic results of  stability data on an abelian category $\cal{A}$ in the sense of \cite{GKR}, and establish the connections between stability data and torsion pairs.
In Section 3, we introduce a partial order relation on the set of stability data. We describe the finest stability data, and introduce a procedure to refine the stability data locally.
In Section 4, we investigate the finest stability data on a tube category, and obtain the classification of torsion pairs via the stability data approach.
In Section 5, we show that each stability data can be refined to a finest one for tame hereditary algebras.
The classification of finest stability data for the projective line and elliptic curves are given in Section 6 and Section 7 respectively.
In the final Section 8, we investigate the stability data for weighted projective lines. In particular, we obtain classification results for finest stability data and torsion pairs for weight type (2).

\emph{Notation.} Throughout this paper, let $\textbf{k}$ be an algebraically closed field, let $\mathcal{A}$ be an abelian category.
All subcategories are assumed to be full and closed under isomorphisms.
For $X, Y \in \mathcal{A}$, we simply denote $\Hom(X, Y):= \Hom_{\mathcal{A}}(X, Y)$ and $\Ext^{i}(X, Y):= \Ext^{i}_{\mathcal{A}}(X, Y)$ for $i\geq 1$.
Given a set $\mathcal{S}$ of objects in $\mathcal{A}$, we write $\langle \mathcal{S} \rangle$ for the subcategory of $\mathcal{A}$ generated by the objects in $\mathcal{S}$ closed \emph{under extensions and direct summands}, and denote by $\add~ \mathcal{S}$ the subcategory of $\mathcal{A}$ consisting of direct summands of finite direct sums of objects in $\mathcal{S}$.
For two subcategories $\cal{T}, \cal{F}$ in $\mathcal{A}$, we denote $\cal{T}^{\perp}:=\{Y\in \mathcal{A} \ |\ \Hom(X, Y)=0, \forall\ X \in \cal{T} \}$ and ${}^{\perp}\cal{F}:=\{X\in \mathcal{A} \ |\ \Hom(X, Y)=0, \forall\ Y \in \cal{F} \}$.
For any $n\in\mathbb{N}$, the set $\Phi =\{i\,|\,1\leq i\leq n\}$ is always viewed as a linearly ordered set in the sense that $1<2<\cdots<n$.

\section{Preliminary}

In this section, we recall the definitions of stability data and torsion pair for an abelian category, and explain the close relations between them.

\subsection{Stability data}
The notion of stability data on an abelian category $\cal{A}$ is introduced by Gorodentsev-Kuleshov-Rudakov in \cite{GKR}.
The most important feature of a stability data is the fact that they create a distinguished subclass of objects in $\cal{A}$ called semistable factors.

\begin{defn}{\rm(}\cite[Def. 2.4]{GKR}{\rm)}\label{defn of stab data}
Suppose that $\cal{A}$ is an abelian category, $\Phi$ is a linearly ordered set, and an extension-closed subcategory
$\Pi_\varphi\subset\cal{A}$ is given for every $\varphi\in\Phi$. A pair
$(\Phi,\{\Pi_\varphi\}_{\varphi\in\Phi})$ is called a \emph{stability data} if
\begin{itemize}
   \item [(1)] ${\rm{Hom}}(\Pi _{\varphi'},\Pi_{\varphi''})=0$ for all $\varphi'>\varphi''$ in $\Phi$;
  \item [(2)] every non-zero object $X\in\mathcal{A}$  has a Harder-Narasimhan filtration
\begin{equation}\label{HN-filt2.1}
\xymatrix @C=6.55em@R=10ex@M=4pt@!0{ 0=X_{0}\ar@{^(->}[r]^-{p_{0}} &X_1\ar@{->>}[d]^{q_{1}}\ar@{^(->}[r]^-{p_{1}}&X_{2}
\ar@{->>}[d]^-{q_{2}}\ar@{^(->}[r]^-{p_{2}}
&\cdots\ar@{^(->}[r]^-{p_{n-2}}&X_{n-1}\ar@{->>}[d]^-{q_{n-1}}
\ar@{^(->}[r]^-{p_{n-1}}& X_n=X\ar@{->>}[d]^{q_{n}}&\\
& A_1&A_{2}&&A_{n-1}&A_{n}\\ }
\end{equation}
with non-zero factors $A_{i}=X_{i}/ X_{i-1}\in\Pi_{\varphi_i}$ and strictly decreasing $\varphi _{i}>\varphi_{i+1}.$
\end{itemize}
\end{defn}

The
categories $\Pi _\varphi$ are called the semistable subcategories of
the stability data $(\Phi,\{\Pi_\varphi\}_{\varphi\in\Phi})$. The nonzero objects in $\Pi _\varphi$ are
said to be semistable of \emph{phase} $\varphi$, while the minimal objects
are said to be stable.
The filtration $\eqref{HN-filt2.1}$ is called the Harder-Narasimhan filtration (\emph{HN-filtration} for short) of $X$, which is unique up to isomorphism. The factors $A_{i}$' s are called the \emph{HN-factors} of $X$, where
$A_1$ is called the \emph{maximal semistable subobject} of $X$, and $A_n$ is called the \emph{minimal semistable quotient object} of $X$.
Define $\bm \phi^{+}(X):=\varphi_1$ and $\bm \phi^{-}(X):=\varphi_n$. Then $X\in\Pi_{\varphi}$ if and
only if {{$\bm \phi^{+}(X)=\bm \phi^{-}(X)=\varphi=:\bm \phi(X)$.}}
For simplification, we often omit the trivial exact sequence  $0\rightarrow X_0\stackrel{0}\longrightarrow  X_1\stackrel{q_{1}}\longrightarrow  A_1\rightarrow0$ in $(2.1)$ in the following.

For an extension-closed subcategory $\cal{B}$ of $\cal{A}$, if there exists $\Pi_\varphi\subset\cal{B}$ for any $\varphi\in\Phi$ satisfying Definition
\ref{defn of stab data} (1) and \eqref{HN-filt2.1} for any non-zero object $X\in\mathcal{B}$, then
we call the pair $(\Phi,\{\Pi_\varphi\}_{\varphi\in\Phi})$ a \emph{local stability data} on $\cal{B}$.

The following observation is important for our later use.

\begin{lem} \label{maxsubss}
Keep the notations in (\ref{HN-filt2.1}), then for each $i$,
\begin{itemize}
  \item [(1)] $X_{i+1}/X_i$ is the maximal semistable subobject of $X/X_i$;
  \item [(2)] $A_i$ is the minimal semistable quotient object of $X_i$.
\end{itemize}
\end{lem}

\begin{proof} {{We only prove the second statement, since the proof for the first one is dual. Assume there is a non-zero semistable quotient object $Y$ of $X_i$ with $\bm \phi(Y)<\bm \phi(A_i)$. Then $\Hom(A_j, Y)=0$ for any $j\leq i$ since $\bm \phi(Y)<\bm \phi(A_i)\leq \bm \phi(A_j)$. Now applying $\Hom(-, Y)$ to the exact sequence $0\to X_{j-1}\xrightarrow{p_{j-1}} X_{j}\xrightarrow{q_{j}} A_{j}\to 0$ for $1<j\leq i$, we obtain that $\Hom(X_{j}, Y)=0$ for any $j\leq i$. But $Y$ is a quotient object of $X_i$, we have $\Hom(X_i, Y)\neq 0$, a contradiction.}}
\end{proof}

\begin{rem} We can obtain the HN-filtration of $X$ by considering the maximal semistable subobject of $X/X_i$ step by step for $i=0, 1, \cdots, n$, or by considering the minimal semistable quotient object of $X_i$ step by step for $i=n, n-1, \cdots, 0$.
\end{rem}

The following lemma shows that each semistable subcategory is closed direct summands.
\begin{lem}
Let $(\Phi,\{\Pi_\varphi\}_{\varphi\in\Phi})$ be a stability data on $\cal{A}$. Then
$$\Pi_\varphi= \big(\bigcap\limits_{\psi>\varphi}\Pi_\psi ^{\perp}\big)\cap \big(\bigcap\limits_{\psi<\varphi} {}^{\perp}\Pi_\psi \big),$$  where
$$\Pi_\psi ^{\perp}=\{X\in \cal{A}\,|\,\Hom(A, X)=0, \;\forall\, A\in \Pi_\psi\}\quad \text{and}\quad {}^{\perp}\Pi_\psi =\{X\in \cal{A}\,|\,\Hom(X, A)=0, \;\forall\, A\in \Pi_\psi\}.$$
Consequently, $\Pi_\varphi$ is closed under direct summands.
\end{lem}

\begin{proof}
By Definition \ref{defn of stab data} (1), we have ${\rm{Hom}}(\Pi _{\varphi'},\Pi_{\varphi''})=0$ for all $\varphi'>\varphi''$ in $\Phi$. Hence $\Pi_\varphi\subseteq \Pi_\psi ^{\perp}$ for any $\psi>\varphi$, and $\Pi_\varphi\subseteq {}^{\perp}\Pi_\psi$ for any $\psi<\varphi$. Therefore, $\Pi_\varphi\subseteq \big(\bigcap\limits_{\psi>\varphi}\Pi_\psi ^{\perp}\big)\cap \big(\bigcap\limits_{\psi<\varphi} {}^{\perp}\Pi_\psi \big)$.

On the other hand, for any $X\in \big(\bigcap\limits_{\psi>\varphi}\Pi_\psi ^{\perp}\big)\cap \big(\bigcap\limits_{\psi<\varphi} {}^{\perp}\Pi_\psi \big)$.
Since $\Hom(\Pi_\psi, X)=0$ for any $\psi>\varphi$, we have $\bm \phi^{+}(X) \leq \varphi$.
Similarly, since $\Hom(X, \Pi_\psi)=0$ for any $\psi<\varphi$, we have $ \bm \phi^{-}(X)\geq\varphi $.
Hence $\bm \phi^{-}(X) =\bm \phi^{+}(X) = \varphi$, and then $X\in \Pi_\varphi$.
Therefore, $\Pi_\varphi\supseteq \big(\bigcap\limits_{\psi>\varphi}\Pi_\psi ^{\perp}\big)\cap \big(\bigcap\limits_{\psi<\varphi} {}^{\perp}\Pi_\psi \big)$.
We are done.
\end{proof}

Given a set $\mathcal{S}$ of objects in $\mathcal{A}$, recall that
$\langle\mathcal{S}\rangle$ denotes the
subcategory of $\mathcal{A}$ generated by the objects in $\mathcal{S}$ closed under extensions and direct summands.
For any interval $I\subseteq \Phi$,
define $\Pi_{I}:=\langle\Pi_{\varphi}\,|\,\varphi\in I\rangle$ to be the subcategory of $\mathcal{A}$ generated by $\Pi_{\varphi}$ for all $\varphi\in I$. Note that non-zero objects in $\Pi_{I}$ consists precisely of those objects  $X\in\cal{A}$ which satisfy $\bm \phi^{\pm}(X)\in I$.

\subsection{Torsion pair}

The notion of torsion pair in an abelian category was first introduced by Dickson \cite{Dickson}, generalizing properties of abelian groups of finite rank. We recall the definition as
follows.

\begin{defn}
A pair $(\cal{T}, \cal{F})$ of subcategories in an abelian category $\cal{A}$ is a \emph{torsion pair} if the following conditions are satisfied:

\begin{itemize}
 \item [(1)]    $\Hom(\cal{T}, \cal{F})=0$;
 \item [(2)] for any $Z\in\cal{A}$, there is an exact sequence $0\to X\to Z\to Y\to 0$ with $X\in\cal{T}, Y\in\cal{F}$.
 \end{itemize}
\end{defn}

By definition, there are two trivial torsion pairs in $\cal{A}$, namely, $(\cal{A},0)$ and $(0, \cal{A})$. All the other torsion pairs are called non-trivial.
A pair $(\cal{T}, \cal{F})$ in $\cal{A}$ is called a \emph{tilting torsion pair} if there is a tilting object $T$ (c.f. \cite{CF}) such that
 $$\cal{T}= \{ E \in \cal{A} \ | \ \Ext^{1}(T, E) = 0\}, \quad \cal{F}= \{ E \in \cal{A} \ | \ \Hom(T, E) = 0 \}.$$

The following result is well-known, see for example
  \cite[Prop. 3.3]{Dickson}, \cite[Lem. 1.1.3]{Po} and \cite[Prop. 5.14]{Tat}.
\begin{prop}\label{torsion pair on Noetherian abelian}
 Let $\mathcal{A}$ be an abelian category. Then a pair $(\cal{T}, \cal{F})$ of subcategories in $\mathcal{A}$ is a torsion pair if and only if $\cal{T} = {}^{\bot}\cal{F}$ and $\cal{F} = \cal{T}^{\bot}$.
\end{prop}

A  torsion pair gives a natural stability data $(\Phi,\{\Pi_{\varphi}\}_{\varphi\in\Phi})$, where $\Phi=\{-, +\}$ with the order $-<+$, $\Pi_{-}=\cal{F}$ and $\Pi_{+}=\cal{T}$.  In fact, a stability data can be thought as a kind of refinement of a torsion pair.

\begin{prop}\label{torsion pair}
 Let $(\Phi,\{\Pi_{\varphi}\}_{\varphi\in\Phi})$ be a stability data on $\mathcal{A}$. Let $\Phi=\Phi_{-}\cup\Phi_{+}$
  be an arbitrary decomposition with $\varphi_-<\varphi_+$ for any $\varphi_{\pm}\in\Phi_{\pm}$. Then the subcategories
$$\cal{T}=\langle\Pi_{\varphi}\mid\varphi\in\Phi_+ \rangle\quad\text{and}\quad\cal{F}=\langle\Pi_{\varphi}\mid\varphi\in\Phi_-\rangle$$
give a torsion pair $(\cal{T},\cal{F})$ in $\cal{A}$.
\end{prop}

\begin{proof} By construction we have $\Hom(\cal{T},\cal{F})=0$. In order to prove the second axiom of a torsion pair, we consider the HN-filtration \eqref{HN-filt2.1} for any non-zero object $X\in\cal{A}$. If $\bm \phi^{-}(X)\in\Phi_+$ or $\bm \phi^{+}(X)\in\Phi_-$, then $X$ belongs to the subcategory $\cal{T}$ or $\cal{F}$ respectively. Otherwise, we find some $k$, such that $\varphi_{k}\in\Phi_+$ but $\varphi_{k+1}\in\Phi_-$.
Then we get the following exact sequence $0\to X_k\to X\to X/X_k\to 0$, where $X_k\in\cal{T}$ and $X/X_k\in\cal{F}$. This finishes the proof.
\end{proof}

As a consequence of the previous proposition, we have the following result which
provides a method to construct torsion pairs in $\mathcal{A}$ using stability data.
\begin{cor}\label{torsion pair determined by phi}
 Let $(\Phi,\{\Pi_{\varphi}\}_{\varphi\in\Phi})$ be a stability data on $\mathcal{A}$. Then each $\psi\in\Phi$ determines two torsion pairs $(\Pi_{\geq\psi},\Pi_{<\psi})$ and $(\Pi_{>\psi},\Pi_{\leq\psi})$, where
$$\Pi_{\geq\psi}=\langle\Pi_{\varphi}\mid\varphi\geq \psi \rangle,\quad\Pi_{<\psi}=\langle\Pi_{\varphi}\mid\varphi<\psi\rangle;$$
$$\Pi_{>\psi}=\langle\Pi_{\varphi}\mid\varphi> \psi \rangle,\quad\Pi_{\leq\psi}=\langle\Pi_{\varphi}\mid\varphi\leq\psi\rangle.$$
\end{cor}

\section{Finest stability data }

In this section, we define a partial ordering for the set of all stability data on an abelian category $\mathcal{A}$. We introduce a procedure to refine a stability data, and give a sufficient and necessary condition to justify a stability data to be finest. Finally, we show that any stability data can be refined to a finest one for representation-directed algebras.

\subsection{Local refinement}

In order to classify the torsion pairs via the stability data approach, we only need to consider the equivalent classes of stability data in the following sense.
\begin{defn} Two stability data $(\Phi,\{\Pi_{\varphi}\}_{\varphi\in\Phi}),(\Psi,\{P_{\psi}\}_{\psi\in\Psi})$ on $\mathcal{A}$ are called \emph{equivalent} if there exists an order-preserved bijective map $r:\Phi\rightarrow\Psi$ such that  $P_{r(\varphi)}=\Pi_\varphi$ for any $\varphi\in\Phi$.
\end{defn}

Now let us define a partial ordering for the set of all stability data on an abelian category $\mathcal{A}$ as follows.

\begin{defn}\label{finer or coarser} Let $(\Phi,\{\Pi_{\varphi}\}_{\varphi\in\Phi}),(\Psi,\{P_{\psi}\}_{\psi\in\Psi})$ be stability data on $\mathcal{A}$. We say that the stability data $(\Psi,\{P_{\psi}\}_{\psi\in\Psi})$ is \emph{finer} than $(\Phi,\{\Pi_{\varphi}\}_{\varphi\in\Phi})$, or  $(\Phi,\{\Pi_{\varphi}\}_{\varphi\in\Phi})$ is \emph{coarser} than $(\Psi,\{P_{\psi}\}_{\psi\in\Psi})$ and write  $(\Phi,\{\Pi_{\varphi}\}_{\varphi\in\Phi})\preceq(\Psi,\{P_{\psi}\}_{\psi\in\Psi})$, if there exists a surjective map $r:\Psi\rightarrow\Phi$ such that
\begin{itemize}
 \item [(1)] $\psi'>\psi''$ implies $ r(\psi')\geq r(\psi'')$;
  \item [(2)] for any $\varphi\in\Phi$, $\Pi_{\varphi}=\langle P_{\psi}\mid\psi\in r^{-1}(\varphi)\rangle$.
\end{itemize}
\end{defn}

In other words, a coarser stability data is obtained from a finer one by fusing certain
blocks of consecutive semistable subcategories. The relation ``finer-coarser"
defines a partial ordering on the set of all stability data on a given abelian category.
Minimal elements with respect to this partial ordering will be called
the \emph{finest} stability data. It is clear that the finest stability data contain the most complete information about the torsion pairs.

For a given stability data on an abelian category, we introduce a local-refinement method as follows.

\begin{prop}\label{local refinement construction} Let $(\Phi,\{\Pi_{\varphi}\}_{\varphi\in\Phi})$ be a stability data on $\mathcal{A}$.
For any $\varphi\in\Phi$, assume $(I_{\varphi},\{P_{{\psi}}\}_{\psi \in I_{\varphi}})$ is a local stability data on $\Pi_{\varphi}$.
Let $\Psi= \cup_{\varphi \in \Phi } I_{\varphi}$, which is a linearly ordered set containing each $I_{\varphi}$ as a linearly ordered subset, and $\psi_1>\psi_2$ whenever $\psi_{i}\in I_{\varphi_i}$ with $\varphi_1>\varphi_2$. Then
$(\Psi, \{P_{\psi}\}_{\psi \in \Psi})$ is a stability data on $\mathcal{A}$, which is finer than $(\Phi,\{\Pi_{\varphi}\}_{\varphi\in\Phi})$.
\end{prop}

\begin{proof}
 By definition we have $\Hom(P_{\psi'}, P_{\psi''})=0$ for any $\psi'>\psi'' \in \Psi$. Moreover, for any non-zero object $X\in\cal{A}$, there is a HN-filtration (\ref{HN-filt2.1}) with respect to the stability data $(\Phi,\{\Pi_{\varphi}\}_{\varphi\in\Phi})$. Since each $(I_{\varphi},\{P_{{\psi}}\}_{\psi \in I_{\varphi}})$ is a local stability data on $\Pi_{\varphi}$, each factor $A_i\in\Pi_{\varphi_{i}}$ has a filtration of the following form
\begin{equation*}
\xymatrix @C=8.25em@R=10ex@M=4pt@!0{ 0=A_{i, 0}\ar@{^(->}[r]^-{f_{i, 0}} &A_{i, 1}\ar@{->>}[d]^{g_{i, 1}}\ar@{^(->}[r]^-{f_{i, 1}}&A_{i, 2}
\ar@{->>}[d]^-{g_{i, 2}}\ar@{^(->}[r]^-{f_{i, 2}}
&\cdots\ar@{^(->}[r]^-{f_{i, m_{i}-2}}&A_{i, m_{i}-1}\ar@{->>}[d]^-{g_{i, m_{i}-1}}
\ar@{^(->}[r]^-{f_{i, m_{i}-1}}& A_{i, m_{i}}=A_{i}\ar@{->>}[d]^{g_{i, m_{i}}}&\\
& B_{i, 1}&B_{i, 2}&&B_{i, m_{i}-1}&B_{i, m_{i}}\\ }
\end{equation*}
with non-zero factors $B_{i, k}=A_{i, k}/ A_{i, k-1}\in P_{\psi_{k}}$ and strictly decreasing $\psi_{k}>\psi_{k+1}\in I_{\varphi_i}.$

By (\ref{HN-filt2.1}) we have a short exact sequence
$\xi_i: 0\rightarrow X_{i-1}\xrightarrow{p_{i-1}} X_{i}\xrightarrow {q_{i}} A_{i}\rightarrow 0 $
for any  $i=1, 2, \cdots, n$. Taking pullback along  $f_{i, m_{i}-1}$,
we obtain the following commutative diagram
$$\xymatrix{
   & & 0 \ar[d]_{}   & 0\ar[d]_{} &  & & \\
  0  \ar[r]^{} & X_{i-1}\ar@{=}[d]_{}  \ar[r]^{} &  X_{i,m_{i}-1} \ar[d]^-{p_{i, m_{i}-1}} \ar[r]^{} & A_{i,m_{i}-1} \ar[d]^-{f_{i, m_{i}-1}}\ar[r]^{} & 0  & &   \\
  0 \ar[r]^{} &X_{i-1}  \ar[r]^-{p_{i-1}} & X_{i}  \ar[r]^-{q_{i}} \ar[d]^-{q_{i, m_{i}}}& A_{i,m_{i}}=A_{i}\ar[d]^{g_{i, m_{i}}} \ar[r]^{} & 0.  & &  \\
   &  & B_{i,m_{i}}\ar[d]_{} \ar@{=}[r]^{} &B_{i,m_{i}} \ar[d]_{} &   & &  \\
    & & 0  & 0 &  & & \\
 }$$
Then we obtain the following two short exact sequences:
$$\xi_{i,m_{i}-1}:0\rightarrow X_{i-1}\rightarrow X_{i,m_{i}-1}\rightarrow A_{i,m_{i}-1}\rightarrow 0 \quad\text{and}\quad 0\rightarrow X_{i,m_{i}-1}\xrightarrow{p_{i, m_{i}-1}} X_{i}\xrightarrow{q_{i, m_{i}}} B_{i,m_{i}}\rightarrow 0, $$
which fit together as follows:
\begin{equation*}
\xymatrix@C=8.75em@R=10ex@M=4pt@!0 { X_{i-1}\ar@{^(->}[r]^-{} &X_{i,m_{i}-1}\ar@{->>}[d]^-{}\ar@{^(->}[r]^-{p_{i, m_{i}-1}}&X_{i,m_{i}}=X_{i}.\ar@{->>}[d]^-{q_{i, m_{i}}}\\
&A_{i,m_{i}-1}&B_{i,m_{i}} }
\end{equation*}
By replacing $\xi_i$ by $\xi_{i,m_{i}-1}$, and taking pullback along $f_{i, m_{i}-2}$,  and keeping the procedure going on step by step, we finally obtain that
each $X_{i}$ admits a finite filtration as follows:
\begin{equation*}
\xymatrix @C=8.48em@R=10ex@M=4pt@!0{ X_{i-1}\ar@{^(->}[r]^-{p_{i, 0}} &X_{i,1}\ar@{->>}[d]^{q_{i, 1}}\ar@{^(->}[r]^-{p_{i, 1}}&X_{i,2}
\ar@{->>}[d]^-{q_{i, 2}}\ar@{^(->}[r]^-{p_{i, 2}}
&\cdots\ar@{^(->}[r]^-{p_{i, m_{i}-2}}&X_{i,m_{i}-1}\ar@{->>}[d]^-{q_{i, m_{i}-1}}
\ar@{^(->}[r]^-{p_{i, m_{i}-1}}& X_{i,m_{i}}=X_{i}.\ar@{->>}[d]^{q_{i, m_{i}}}&\\
& B_{i, 1}&B_{i, 2}&&B_{i, m_{i}-1}&B_{i, m_{i}}\\ }
\end{equation*}
Therefore, $X$ admits a subobject filtration with factors $(B_{1,1}, \cdots, B_{1,m_{1}},  B_{2,1}, \cdots, B_{2,m_{2}}, \cdots, B_{n,1}, \cdots,$ $ B_{n,m_{n}})$, having strictly decreasing order by the definition of ordering in $\Psi$. Hence the subobject filtration of $X$ is in fact a HN-filtration, which ensures that $(\Psi, \{P_{\psi}\}_{\psi \in \Psi})$ is a stability data on $\cal{A}$.

Recall that $\Psi= \cup_{\varphi \in \Phi } I_{\varphi}$. This induces a well-defined surjective map $r: \Psi\to \Phi$ such that $r(\psi)=\varphi$ for any $ \psi \in I_{\varphi}$. In order to show $(\Psi, \{P_{\psi}\}_{\psi \in \Psi})$ is finer than $(\Phi,\{\Pi_{\varphi}\}_{\varphi\in\Phi})$, it remains to show the statements (1) and (2) in
Definition \ref{finer or coarser} hold. In fact,
if $\psi_{1}>\psi_{2} \in\Psi$, we claim $r(\psi_{1})\geq r(\psi_{2})$. Otherwise, we write $r(\psi_{i})=\varphi_{i}$, that is $\psi_{i} \in I_{\varphi_{i}}$ for $i=1,2$. Then $\varphi_{1}<\varphi_{2}$, which follows that $\psi_{1}<\psi_{2}$ by definition of ordering in $\Psi$, a contradiction. This proves the claim.
On the other hand, for any $\varphi\in\Phi$,
$$\Pi_{\varphi}=\langle P_{{\psi}}\ |\ \psi \in I_{\varphi}\rangle=\langle P_{{\psi}}\ |\ \psi \in r^{-1} (\varphi)\rangle.$$
We are done.
\end{proof}

The stability data $(\Psi, \{P_{\psi}\}_{\psi \in \Psi})$ obtained in the above way will be called a \emph{local refinement} of $(\Phi, \{\Pi_{\varphi}\}_{\varphi\in\Phi})$.

\subsection{Finest stability data}

In this subsection, we provide a criterion for a stability data to be finest on arbitrary abelian category.

\begin{thm}\label{iff cond for finest}
Let $\mathcal{A}$ be an abelian category.  A stability data $(\Phi,\{\Pi_{\varphi}\}_{\varphi\in\Phi})$ on $\mathcal{A}$ is finest if and only if for any $\varphi\in\Phi$ and non-zero objects $X,Y\in\Pi_{\varphi}$, $\Hom(X,Y)\neq 0\neq \Hom(Y,X)$.
\end{thm}

\begin{proof}  To prove the `` if " part, we assume $(\Phi,\{\Pi_{\varphi}\}_{\varphi\in\Phi})$  is not finest. Then there exists a stability data $(\Psi, \{P_{\psi}\}_{\psi \in \Psi})$ which is finer than $(\Phi,\{\Pi_{\varphi}\}_{\varphi\in\Phi})$. Hence there is a surjective map $r: \Psi \rightarrow \Phi$, which is not a bijection. So there exists $\varphi \in \Phi$ and $\psi_1> \psi_2\in\Psi$ such that $r(\psi_1) = r(\psi_2)=\varphi$. Then $P_{\psi_1}, P_{\psi_2} \subseteq \Pi_{\varphi}$, but $\Hom(P_{\psi_1}, P_{\psi_2})=0$, a contradiction.

For the `` only if " part, we assume there exists some $\varphi\in\Phi$ and non-zero $X,Y\in\Pi_{\varphi}$ such that $\Hom(X,Y)=0$. We will show that $(\Phi,\{\Pi_{\varphi}\}_{\varphi\in\Phi})$ is not finest.

For this we first need to construct a new torsion pair in $\mathcal{A}$.
Define
$$\Pi_{\varphi_{-}}=\{Z\in\Pi_{\varphi}\mid \Hom(X,Z)=0\}\quad \text{and\quad} \Pi_{\varphi_{+}}=\{W\in\Pi_{\varphi}\mid \Hom(W,Z)=0, \forall Z\in \Pi_{\varphi_{-}}\}.$$
Then obviously $Y\in\Pi_{\varphi_{-}}$,
$X\in\Pi_{\varphi_{+}}$ and $\Hom(\Pi_{\varphi_{+}}, \Pi_{\varphi_{-}})=0$.
Let $$\cal{T}=\langle  \Pi_{\psi}, \psi>\varphi; \; \Pi_{\varphi_{+}}\rangle\quad \text{and\quad} \cal{F}=\langle  \Pi_{\psi}, \psi<\varphi; \; \Pi_{\varphi_{-}}\rangle.$$ We claim that
$(\cal{T},\cal{F})$ forms a torsion pair in
$\mathcal{A}$. According to Proposition \ref{torsion pair on Noetherian abelian}, it suffices to show that $\cal{T}={}^{\bot}\cal{F}$ and $\cal{F}=\cal{T}^{\bot}$. We only show the first statement, the second one can be obtained similarly.

By definition, we have $\Hom(\cal{T},\cal{F})=0$. Hence $\cal{T}\subseteq{}^{\bot}\cal{F}$. On the other hand, for any
$Z\in {}^{\bot}\cal{F}$, we need to show that $Z\in \cal{T}$.
Assume the HN-filtration of $Z$ is given by
\begin{equation*}
\xymatrix @C=6.55em@R=10ex@M=4pt@!0{ 0=Z_{0}\ar@{^(->}[r]^-{p_{0}} &Z_1\ar@{->>}[d]^{q_{1}}\ar@{^(->}[r]^-{p_{1}}&Z_{2}
\ar@{->>}[d]^-{q_{2}}\ar@{^(->}[r]^-{p_{2}}
&\cdots\ar@{^(->}[r]^-{p_{n-2}}&Z_{n-1}\ar@{->>}[d]^-{q_{n-1}}
\ar@{^(->}[r]^-{p_{n-1}}& Z_n=Z\ar@{->>}[d]^{q_{n}}&\\
& A_1&A_{2}&&A_{n-1}&A_{n}\\ }
\end{equation*}
with factors $A_i\in\Pi_{\varphi_i}$.
Now $\Hom(Z,\cal{F})=0$ implies $\varphi_{n}\geq \varphi$. If $\varphi_{n}> \varphi$, then $Z\in\langle\Pi_{\psi}\ |\ \psi>\varphi\rangle\subseteq \cal{T}$.
If $\varphi_{n}= \varphi$, then $\Hom(Z, \Pi_{\varphi_{-}})=0$ implies  $\Hom(A_n, \Pi_{\varphi_{-}})=0$ since there is a surjection $q_{n}: Z\twoheadrightarrow A_n$.
Hence $A_n\in({}^{\bot}\Pi_{\varphi_{-}})\cap \Pi_{\varphi}=\Pi_{\varphi_{+}}\subseteq \cal{T}$ and then $Z\in \cal{T}$.
Hence, $\cal{T}={}^{\bot}\cal{F}$.

Since $(\cal{T},\cal{F})$ is a torsion pair, for any object $W\in \Pi_{\varphi} \subset \cal{A}$, there exists a decomposition
$0\to W_t\to W\to W_f\to 0$ with $W_t\in\cal{T}$ and $W_f\in\cal{F}$.
For any $\psi>\varphi$, $\Hom(\Pi_{\psi}, W)=0$ implies $\Hom(\Pi_{\psi}, W_t)=0$, which ensures that $W_t\in \langle \Pi_{\psi}\ |\ \psi\leq\varphi\rangle\cap \cal{T}=\Pi_{\varphi_{+}}$.
Similarly, one can show that $W_f\in \Pi_{\varphi_{-}}$.
Therefore, $\Pi_{\varphi}=\langle \Pi_{\varphi_{-}},  \Pi_{\varphi_{+}}\rangle$.

Let $\Psi=(\Phi\setminus\{\varphi\}) \cup \{\varphi_{-}, \varphi_{+}\}$, which is a linearly ordered set with the relations $\varphi'>\varphi_{+}>\varphi_{-}>\varphi''$ for any $\varphi'>\varphi_{}>\varphi'' \in \Phi$.
By Proposition \ref{local refinement construction}, we know that
$(\Psi, \{\Pi_{\psi}\}_{\psi \in \Psi})$ is a stability data, which is finer than $(\Phi,\{\Pi_{\varphi}\}_{\varphi\in\Phi})$, yielding a contradiction.
We are done.
\end{proof}

We know that stability data is a refinement of torsion pairs. In fact, for certain cases, all the torsion pairs can be obtained from finest stability data.

\begin{prop}\label{torsion pair construction}
Assume that each stability data on $\cal{A}$ is coarser than a finest one. Then for any torsion pair $(\cal{T}, \cal{F})$, there exists a stability data $(\Phi,\{\Pi_{\varphi}\}_{\varphi\in\Phi})$ and a decomposition  $\Phi=\Phi_{-}\cup\Phi_{+}$, such that $$\cal{T}=\langle\Pi_{\varphi}\mid\varphi\in\Phi_+ \rangle,\quad\cal{F}=\langle\Pi_{\varphi}\mid\varphi\in\Phi_-\rangle.$$
Moreover, $(\Phi,\{\Pi_{\varphi}\}_{\varphi\in\Phi})$ can be taken from the set of finest stability data.
\end{prop}

\begin{proof} Recall that each  torsion pair gives a stability data. Then the result follows from Proposition \ref{torsion pair}.
\end{proof}

Let $A$ be a representation-directed algebra. Recall from \cite{Ringel} that the indecomposable modules in $\mod$-$A$ can be ordered by a linearly ordered set $\Phi_{A}$, namely, $\ind (\mod$-$A)$ $=\{M_{\varphi} \ |\ \varphi\in \Phi_{A}\}$, such that $\Hom(M_{\varphi'},M_{\varphi''})=0$ for any $\varphi'>\varphi''\in \Phi_{A}$. Let $\Pi_{\varphi}=\langle M_{\varphi}\rangle$ for any $\varphi\in\Phi_A$, then by Theorem \ref{iff cond for finest}, $(\Phi_A,\{\Pi_{\varphi}\}_{\varphi\in\Phi_A})$ is a finest stability data on $\mod$-$A$. Clearly, any finest stability data on $\mod$-$A$ has the above form (depends on $\Phi_A$), and any stability data can be refined to a finest one.

Note that the path algebra $\mathbf{k}Q$ of a Dynkin quiver $Q$ is a representation-directed algebra, hence each stability data on $\mod$-$\mathbf{k}Q$ is coarser than a finest one. By Proposition \ref{torsion pair construction}, we can obtain the classification of torsion pairs in the module category $\mod$-$\mathbf{k}Q$ via the stability data approach. In the following we provide two concrete examples for path algebras of Dynkin type $A_2$ and $A_3$.

\begin{exam}
Let $Q:1 \rightarrow  2$ and $A_2$:=$\mathbf{k}Q$. Then there are only three indecomposable modules in $\mod$-$A_2$, the simple modules $S_1, S_2$ and the projective modules $P_1, P_2(=S_2)$. There are only two equivalent classes of finest stability data $(\Phi,\{\Pi_{\varphi}\}_{\varphi\in\Phi})$ on $\mod$-$A_2$ as follows:
\begin{itemize}
  \item [(1)]  $\Phi =\{ 1, 2, 3  \}$, and\ $\Pi_{ 1}=\langle S_2\rangle, \Pi_{2}=\langle P_1\rangle, \Pi_{ 3}=\langle S_1\rangle$;
in this case, each indecomposable module is semistable;
  \item [(2)] $\Phi =\{ 1, 2  \}$, and\ $\Pi_{ 1}=\langle S_1\rangle,  \Pi_{2}=\langle S_2\rangle$;
in this case, only $S_1, S_2$ are semistable indecomposable modules.  The indecomposable module $P_1$ is not semistable, whose HN-filtration is induced from the exact sequence
$$0\rightarrow S_2 \rightarrow P_1 \rightarrow S_1 \rightarrow0.$$
\end{itemize}

As a consequence of  Proposition \ref{torsion pair construction}, we obtain a classification of torsion pairs in $\mod$-$A_2$ as below:

\begin{table}[h]
\caption{Non-trivial torsion pairs in $\mod$-$A_2$}
\begin{tabular}{|c|c|c|c|c|}
\hline
\makecell*[c]{$\cal{T}$}&$\cal{F}$\\
\hline
\makecell*[c]{$\langle S_1 \rangle$}&$ \langle P_1,S_2 \rangle $\\
\hline
\makecell*[c]{$\langle  S_{2} \rangle$}&$ \langle S_{1} \rangle $\\
\hline
\makecell*[c]{$\langle P_1, S_1  \rangle  $}&$  \langle  S_2 \rangle$\\
\hline
\end{tabular}
\label{table 2}
\end{table}
\end{exam}

\begin{exam}
Let $Q: 1\rightarrow2\rightarrow3$ and $A_3$:=$\mathbf{k}Q$.
Then the Auslander-Reiten quiver $\Gamma(\mod$-$A_3)$ of the module category $\mod$-$A_3$ has the form
$$
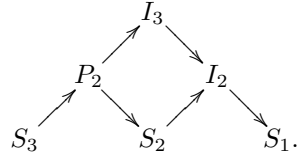
\xymatrix@M=1pt@!0{
   &  & I_{3} \ar[dr]_{}  &  &  \\
   & P_{2}\ar[dr]_{} \ar[ur]^{} &  & I_{2}\ar[dr]_{}  & \\
  S_{3} \ar[ur]^{} &  & S_{2} \ar[ur]^{} &  & S_{1}.  }$$

Let $\Phi =\{ 1, 2, 3\}$ be a linear ordered set in the natural sense $1<2<3$, and let $\Pi_{1}=\langle S_1\rangle, \Pi_{2}=\langle S_2\rangle, \Pi_{3}=\langle S_3\rangle$, then $(\Phi,  \{\Pi_{\varphi}\}_{\varphi\,\in\, \Phi})$ is a finest stability data on $\mod$-$A_3$.
In this case, $\bm \phi(S_{3})>\bm \phi(S_{2})>\bm \phi(S_{1})$ and only $S_3, S_2, S_1$ are semistable indecomposable modules. The indecomposable module $P_2, I_{3}, I_{2}$ are not semistable, whose HN-filtration are given as below respectively:
\begin{equation*}
{\footnotesize\xymatrix { S_{3}^{}\ar@{^(->}[r]^-{} &P_{2}\ar@{->>}[d]^{}&&
S_{2}^{}\ar@{^(->}[r]^-{} &P_{2}\ar@{->>}[d]^{}\ar@{^(->}[r]^-{} &I_{3}\ar@{->>}[d]^{}&&
S_{2}^{}\ar@{^(->}[r]^-{}&I_{2}.\ar@{->>}[d]^{}&&\\
&S_{2}&&&S_{2}&S_{1}&&&S_{1} & }}
\end{equation*}

Similarly, we can obtain all the other finest stability data $(\Phi,  \{\Pi_{\varphi}\}_{\varphi\,\in\, \Phi})$ on $\mod$-$A_3$ as follows:

\begin{itemize}
\item[(1)] $\Phi =\{ 1, 2, 3, 4\}$;

\begin{table}[h]
\begin{tabular}{|c|c|c|c|}
\hline
\makecell*[c]{$\Pi_{1}$}&$\Pi_{2}$&$\Pi_{3}$&$\Pi_{4}$\\
\hline
\makecell*[c]{$\langle S_1\rangle$}&$ \langle S_3\rangle $&$\langle P_2\rangle$&$\langle  S_2\rangle$\\
\hline
\makecell*[c]{$\langle S_2\rangle$}&$ \langle I_2\rangle $&$\langle S_1\rangle$&$\langle  S_3\rangle$\\
\hline
\makecell*[c]{$\langle S_2\rangle$}&$ \langle I_2\rangle $&$\langle S_3\rangle$&$\langle  S_1\rangle$\\
\hline
\makecell*[c]{$\langle S_3\rangle$}&$ \langle S_1\rangle $&$\langle P_2\rangle$&$\langle  S_2\rangle$\\
\hline
\end{tabular}
\label{I4}
\end{table}

\item[(2)] $\Phi =\{ 1, 2, 3, 4, 5 \}$;

\begin{table}[h]
\begin{tabular}{|c|c|c|c|c|}
\hline
\makecell*[c]{$\Pi_{1}$}&$\Pi_{2}$&$\Pi_{3}$&$\Pi_{4}$&$\Pi_{5}$\\
\hline
\makecell*[c]{$\langle S_2\rangle$}&$\langle S_3\rangle$&$\langle I_3\rangle$&$\langle  I_2\rangle$&$\langle S_1\rangle$\\
\hline
\makecell*[c]{$\langle S_3\rangle$}&$\langle P_2\rangle$&$\langle I_3\rangle$&$\langle  S_1\rangle$&$\langle S_2\rangle$\\
\hline
\end{tabular}
\label{I5}
\end{table}

\item[(3)] $\Phi =\{ 1, 2, 3, 4, 5, 6\}$;

\begin{table}[h]
\begin{tabular}{|c|c|c|c|c|c|}
\hline
\makecell*[c]{$\Pi_{1}$}&$\Pi_{2}$&$\Pi_{3}$&$\Pi_{4}$&$\Pi_{5}$&$\Pi_{6}$\\
\hline
\makecell*[c]{$\langle S_3\rangle$}&$\langle P_2\rangle$&$\langle S_2\rangle$&$\langle I_3\rangle$&$\langle I_2\rangle$&$\langle S_1\rangle$\\
\hline
\makecell*[c]{$\langle S_3\rangle$}&$\langle P_2\rangle$&$\langle I_3\rangle$&$\langle S_2\rangle$&$\langle I_2\rangle$&$\langle S_1\rangle$\\
\hline
\end{tabular}
\label{I6}
\end{table}

As a consequence of Proposition \ref{torsion pair construction}, we obtain a classification of torsion pairs in $\mod$-$A_3$ as below:

\begin{table}[h]
\caption{Non-trivial torsion pairs $(\mathcal{T},\mathcal{F})$ in $\mod$-$A_3$}
\begin{tabular}{|c|c|c|c|c|c|}
\hline
\makecell*[c]{$\mathcal{T}$}&$\mathcal{F}$& \multirow{4}{*}{} &$\mathcal{T}$&$\mathcal{F}$\\
\cline{1-2}   \cline{4-5}

\makecell*[c]{$\langle S_1\rangle$}&$ \langle S_3, S_2, I_2\rangle $&&$\langle I_2, S_1\rangle$&$\langle S_3, S_2, I_3\rangle$\\
\cline{1-2}   \cline{4-5}
\makecell*[c]{$\langle S_2\rangle$}&$\langle S_3, P_2, S_1\rangle$&&
$\langle P_2, S_2\rangle$&$\langle S_3,  S_1\rangle$\\
\cline{1-2}   \cline{4-5}
\makecell*[c]{$\langle S_3\rangle $}&$\langle S_2, S_1\rangle$&&
$\langle  I_3, I_2, S_1\rangle$&$ \langle S_3, S_2\rangle $\\
\cline{1-2}   \cline{4-5}
\makecell*[c]{$\langle S_2, S_1\rangle$}&$\langle S_3, P_2, I_3\rangle$&&$\langle S_2, I_3,  S_1\rangle$
&$\langle P_2, S_3\rangle$\\
\cline{1-2}   \cline{4-5}
\makecell*[c]{$\langle S_3, S_1 \rangle$}&$\langle S_2, I_2\rangle$&&
$\langle  P_2, S_2, S_1 \rangle$&$ \langle  S_3 \rangle $\\
\cline{1-2}   \cline{4-5}
\makecell*[c]{$\langle S_2, S_3\rangle$}&$ \langle S_1 \rangle$&&
$\langle S_3, I_2, S_1\rangle$&$\langle S_2 \rangle$\\
\hline
\end{tabular}
\label{torsion pairs in A3}
\end{table}
\end{itemize}
\end{exam}

\section{\rm{\textbf{ Finest stability data on tube categories}}}

In this section we focus on the nilpotent representation category of cyclic quiver $C_{n}$ with $n$ vertices, that is, on the  \emph{tube}  category $\textbf{T}_n$ of rank $n\geq1$ in the sense of \cite[Sect. 4.6]{Ringel} and \cite[Chap. X]{SimSkow}, whose associated Auslander-Reiten quiver (AR-quiver for short) $\Gamma(\textbf{T}_n)$ is a \emph{stable tube} of rank $n$. We investigate  the stability data for tube category $\textbf{T}_n$, and provide a new method to classify torsion pairs on $\textbf{T}_n$ via stability data approach.

\subsection{Homological properties for tubes}

 Recall that  $\textbf{T}_n$ is a hereditary finite length abelian category with $n$ simple objects $S_{0}, \cdots, S_{n-1}$, equipped with an Auslander-Reiten translation $\tau$ satisfying $\tau(S_{i} )=S_{i-1}$, where the index is (always) considered module $n$.
A stable tube is called \emph{homogeneous} if it has rank
one and it is called \emph{non-homogeneous} otherwise.
For any simple object $S_{j}$ in the tube category
$\textbf{T}_n$ and any $t\in \mathbb{Z}_{\geq 1}$, there is a unique object $S^{(t)}
_{j}$ of length $t$ and top $(S^{(t)}_{j}) = S_{j}$, and any indecomposable object in $\textbf{T}_n$ has this form. The tube category $\textbf{T}_n$ is a uniserial category in the sense
that all subobjects of $S^{(t)}_{j}$ form a chain with respect to the inclusion:
$$ 0:=S^{(0)}_{j-t}\subseteq S^{(1)}_{j-t+1}\subseteq S^{(2)}_{j-t+2}\subseteq \cdots \subseteq S^{(t-1)}_{j-1}\subseteq S^{(t)}_{j}.$$
Consequently, the socle of $S^{(t)}_{j}$ is given by soc $(S^{(t)}_{j})=S^{}_{j-t+1}$. Moreover, $S^{(r)}_{j-t+r}/S^{(r-1)}_{j-t+r-1}=S^{}_{j-t+r}$ for $1\leq r\leq t$, and the \emph{composition factor sequence} of $S_{j}^{(t)}$ is given by $(S_{j-t+1}, \cdots, S_{j-1}, S_{j})$. The set of pairwise distinct simple objects appearing in the sequence is called the \emph{composition factor set} of $S_{j}^{(t)}$. It is easy to see that any $S_{j}^{(t)}$ with $t\geq n$ has the same composition factor set $\{S_0, S_1, \cdots, S_{n-1}\}$.

We denote by $u_{j,t}: S_{j-1}^{(t-1)} \rightarrow S_{j}^{(t)}$ and $p_{j,t}:  S_{j}^{(t)} \rightarrow S_{j}^{(t-1)}$ the irreducible injection map and surjection map respectively. For convenience, in the following we will simply denote them by $u$ and $p$
respectively if no confusions appear. For example, the composition
$$\xymatrix@C=3.5em@R=5ex@M=4pt@!0{
 & &S_{j}^{(t)} \ar[rr]^-{p_{j,t}}& &S_{j}^{(t-1)} \ar[rr]^-{p_{j,t-1}}& &S_{j}^{(t-2)}\ar[rr]^-{p_{j,t-2}} &&\cdots\ar[rr]^-{p_{j,3}}&&S_{j}^{(2)}\ar[rr]^-{p_{j,2}}& & S_{j}^{}  \\ }
$$
will be just denoted by $p^{t-1}: S_{j}^{(t)} \rightarrow S_{j}^{}$. With this simplified notations, we have $u\circ p = p\circ u$
whenever it makes sense. We have the following fundamental exact sequences in the tube category $\textbf{T}_n$, see for example \cite[Thm. 2.2]{SimSkow} and \cite[Lem. A.1]{RW}.
\begin{lem}\label{exact seqs in tube}
The following are exact sequences in {\rm{\textbf{T}}$_n$} for any $j \in \mathbb{Z}/n\mathbb{Z}$ and $t \geq 1$:
$$\xymatrix@C=4.5em@R=5ex@M=3.5pt@!0{
(1)\ \ 0 \ar[r]& S_{j-1}^{(t)} \ar[rr]^{u}& &S_{j}^{(t+1)} \ar[rr]^-{p^{t}}& & S_{j}^{}\ar[r]& 0; & & & & & & & & &  & & & & & & \\
 (2)\ \ 0 \ar[r]& S_{j-t}^{} \ar[rr]^-{u^{t}}& &S_{j}^{(t+1)} \ar[rr]^-{p}& & S_{j}^{(t)}\ar[r]& 0; & & & & & & & & &  & & & & & & \\
 (3)\ \ 0 \ar[r]& S_{j}^{(t)} \ar[rr]^-{(u,p)^{t}}&  &S_{j+1}^{(t+1)}\oplus S_{j}^{(t-1)} \ar[rr]^-{(p,-u)}&  & S_{j+1}^{(t)}\ar[r]& 0; & & & & & & & & &  & &  \\
 (4)\ \ 0 \ar[r]^-{} &  S_{j-t+1}^{} \ar[rr]^-{u^{t-1}}& &S_{j}^{(t)} \ar[rr]^-{u\circ p = p\circ u}& & S_{j+1}^{(t)}\ar[r]^-{p^{t-1}}  & S_{j+1}^{}\ar[r] & 0.& & & \\
  }
$$
\end{lem}

In the following we consider the non-zero morphisms in the tube category \rm{\textbf{T}}$_n$.

\begin{prop}\label{non-zero morphism by soc or top between objects} For any objects $S_{j_i}^{(t_i)}\; (i=1,2)$ in {\rm{\textbf{T}}$_n$},
\begin{itemize}
  \item [(1)] if $t_1\geq t_2$, then $\Hom(S_{j_1}^{(t_1)}, S_{j_2}^{(t_2)})\neq 0$ if and only if $\top(S_{j_1}^{(t_1)})$ belongs to the composition factors set of $S_{j_2}^{(t_2)}$;
  \item [(2)] if $t_1\leq t_2$, then $\Hom(S_{j_1}^{(t_1)}, S_{j_2}^{(t_2)})\neq 0$ if and only if $\soc(S_{j_2}^{(t_2)})$ belongs to the composition factors set of $S_{j_1}^{(t_1)}$.
\end{itemize}
\end{prop}

\begin{proof}
We only prove the second statement, since the proof for the first one is dual.
First we assume there exists $0\not= f \in \Hom(S_{j_1}^{(t_1)}, S_{j_2}^{(t_2)})$. Then there is a decomposition of  $f: S_{j_1}^{(t_1)} \twoheadrightarrow\im f \hookrightarrow S_{j_2}^{(t_2)}$.
It follows that $\soc(S_{j_2}^{(t_2)})=\soc(\im f)$, which is a composition factor of $\im f$, and hence a composition factor of $S_{j_1}^{(t_1)}$.

On the other hand, assume $\soc(S_{j_2}^{(t_2)})$ belongs to the composition factors set of  $S_{j_1}^{(t_1)}$.
Let $S_{j_1}^{(t)}$ with $t\leq n$ be the unique object such that  $\soc(S_{j_1}^{(t)})=\soc(S_{j_2}^{(t_2)})$. By assumption, $\soc(S_{j_2}^{(t_2)})$ belongs to the composition factors set of $S_{j_1}^{(t_1)}$, we get $t\leq t_1$, hence there is an epimorphism $ S_{j_1}^{(t_1)}\twoheadrightarrow S_{j_1}^{(t)}$. Since $t\leq t_1\leq t_2$ and $\soc(S_{j_1}^{(t)})=\soc(S_{j_2}^{(t_2)})$, there is a monomorphism $S_{j_1}^{(t)}\hookrightarrow  S_{j_2}^{(t_2)}$.
Therefore, there is a non-zero morphism $ S_{j_1}^{(t_1)}\twoheadrightarrow S_{j_1}^{(t)}\hookrightarrow  S_{j_2}^{(t_2)}$. Hence $\Hom(S_{j_1}^{(t_1)}, S_{j_2}^{(t_2)})\neq 0$.
\end{proof}

\begin{prop}\label{non-zero morphism between objects} For any objects $S_{j_i}^{(t_i)}\; (i=1,2)$ in {\rm{\textbf{T}}$_n$}, $\Hom(S_{j_1}^{(t_1)}, S_{j_2}^{(t_2)})\neq 0$ if and only if the following hold:
\begin{itemize}
  \item [(1)] $\top(S_{j_1}^{(t_1)})$ belongs to the composition factors set of $S_{j_2}^{(t_2)}$;
  \item [(2)] $\soc(S_{j_2}^{(t_2)})$ belongs to the composition factors set of $S_{j_1}^{(t_1)}$.
\end{itemize}
\end{prop}

\begin{proof}
First we assume there exists $0\not= f \in \Hom(S_{j_1}^{(t_1)}, S_{j_2}^{(t_2)})$. Then there is a decomposition of  $f: S_{j_1}^{(t_1)} \twoheadrightarrow\im f \hookrightarrow S_{j_2}^{(t_2)}$.
It follows that $\top(S_{j_1}^{(t_1)})=\top(\im f)$, which is a composition factor of $\im f$, and hence a composition factor of $S_{j_2}^{(t_2)}$. Similarly, $\soc(S_{j_2}^{(t_2)})=\soc(\im f)$, which is a composition factor of $\im f$, and hence a composition factor of $S_{j_1}^{(t_1)}$. Hence (1) and (2) hold.

On the other hand, if $t_1\geq t_2$, since $\top(S_{j_1}^{(t_1)})$ belongs to the composition factors set of $S_{j_2}^{(t_2)}$,  by Proposition \ref{non-zero morphism by soc or top between objects} (1), we get $\Hom(S_{j_1}^{(t_1)}, S_{j_2}^{(t_2)})\neq 0$; otherwise,  $t_1< t_2$, since $\soc(S_{j_2}^{(t_2)})$ belongs to the composition factors set of $S_{j_1}^{(t_1)}$,  by Proposition \ref{non-zero morphism by soc or top between objects} (2), we get $\Hom(S_{j_1}^{(t_1)}, S_{j_2}^{(t_2)})\neq 0$. Then we are done.
\end{proof}

\begin{cor}\label{non-zero morphism between heigh objects} For any objects $S_{j_i}^{(t_i)}\; (i=1,2)$ in {\rm{\textbf{T}}$_n$} with $t_1\geq n$,
\begin{itemize}
  \item [(1)] $\Hom(S_{j_1}^{(t_1)}, S_{j_2}^{(t_2)})\neq 0$ if and only if $S_{j_1}$ belongs to the composition factors set of $S_{j_2}^{(t_2)}$;
  \item [(2)] $\Hom(S_{j_2}^{(t_2)}, S_{j_1}^{(t_1)})\neq 0$
  if and only if $S_{j_1-t_1+1}$ belongs to the composition factors set of $S_{j_2}^{(t_2)}$.
\end{itemize}
 In particular, if $t_2\geq n$, then $\Hom(S_{j_1}^{(t_1)}, S_{j_2}^{(t_2)})\neq 0\neq \Hom(S_{j_2}^{(t_2)}, S_{j_1}^{(t_1)})$.
\end{cor}

\begin{proof} The assumption $t_1\geq n$ implies that the composition factor set of $S_{j_1}^{(t_1)}$ is $\{S_0, S_1, \cdots, S_{n-1}\}$, which already contains $\top(S_{j_2}^{(t_2)})$ and $\soc(S_{j_2}^{(t_2)})$. Observe that $\soc(S_{j_1}^{(t_1)})=S_{j_1-t_1+1}$ and $\top(S_{j_1}^{(t_1)})=S_{j_1}$, then the result follows from Proposition \ref{non-zero morphism between objects}.
\end{proof}

The following result plays a key role in this section.

\begin{lem}\label{finitely many extension closed subcategory} There are only finitely many subcategories of {\rm{\textbf{T}}}$_n$ which are closed under extensions and direct summands.
\end{lem}

\begin{proof}
For any $0\leq j\leq n-1$ and $1\leq l\leq n$, using Lemma \ref{exact seqs in tube}, one obtains the following exact sequence
$$0\to S_{j}^{((r-1)n+l)}\to S_{j}^{(rn+l)}\to S_{j}^{(n)}\to 0  \quad(r\geq 1).$$
Then the following pushout-pullback commutative diagram
$$\xymatrix{
 0\ar[r] &S_{j}^{(rn+l)} \ar[r]\ar[d]& S_{j}^{((r+1)n+l)} \ar[r]\ar[d]& S_{j}^{(n)} \ar@{=}[d] \ar[r]& 0\\
 0\ar[r] & S_{j}^{((r-1)n+l)} \ar[r]&S_{j}^{(rn+l)}\ar[r]& S_{j}^{(n)} \ar[r]& 0 }
$$
yields a short exact sequence
$$0\to S_{j}^{(rn+l)}\to S_{j}^{((r-1)n+l)}\oplus S_{j}^{((r+1)n+l)}\to S_{j}^{(rn+l)}\to 0.$$
Hence $S_{j}^{((r\pm1)n+l)}\in \langle S_{j}^{(rn+l)}\rangle$.
Then for $r\geq 1$, we obtain that $\langle S_{j}^{(rn)}\rangle=\langle S_{j}^{(n)}\rangle$, and $\langle S_{j}^{(rn+l)}\rangle=\langle S_{j}^{(n+l)}\rangle$ for $1\leq l\leq n-1$. Note that there are only finitely many indecomposable objects in ${\rm{\textbf{T}}}_n$ with length smaller than $2n$. Then the result follows.
\end{proof}

\subsection{Finest stability data on the  tube categories}
In this subsection, we will describe the semistable subcategories for finest stability data on the tube category \rm{\textbf{T}}$_n$.

\begin{prop}\label{finest stab data for cyclic quiver} Any stability data on {\rm{\textbf{T}}}$_n$ can be refined to a finest one.
\end{prop}

\begin{proof} Let $(\Phi,\{\Pi_{\varphi}\}_{\varphi\in\Phi})$ be a stability data on $\textbf{T}_n$. If it is not finest, then we can make local refinement to obtain a finer stability data. But there are only finitely many ways to make refinement since there are only finitely many candidates for semistable subcategories of $\textbf{T}_n$ by Lemma
\ref{finitely many extension closed subcategory}. Hence the local refinement procedures will stop after finite steps, which yields a finest stability data.
\end{proof}

From now onward, we always fix a finest stability data
$(\Phi,\{\Pi_{\varphi}\}_{\varphi\in\Phi})$ on $\textbf{T}_n$.
Then we have the following results.

\begin{lem} \label{different semistable subcategories}
 For any two distinct semistable objects $S_{j_i}^{(t_i)}$ in {\rm{\textbf{T}}$_n$} with $t_i\leq n,\;i=1,2$, we have $\bm \phi(S_{j_1}^{(t_1)})\neq\bm \phi(S_{j_2}^{(t_2)})$.
 \end{lem}

\begin{proof}
For contradiction we assume there exist $S_{j_i}^{(t_i)}$ with $t_i\leq n \;(i=1,2)$ belong to the same semistable subcategory $\Pi_{\varphi}$ for some $\varphi\in\Phi$. By Theorem \ref{iff cond for finest}, we have
\begin{equation}\label{non zero homo between two}\Hom(S_{j_1}^{(t_1)}, S_{j_2}^{(t_2)})\neq 0\neq \Hom(S_{j_2}^{(t_2)}, S_{j_1}^{(t_1)}).
  \end{equation}
Let $0\not= f \in \Hom(S_{j_2}^{(t_2)}, S_{j_1}^{(t_1)})$. Then $\im f=S_{j_2}^{(t)}$ for some $t\leq \min\{t_1, t_2\}$.
If $t=t_1$, then $\soc(S_{j_2}^{(t)})=\soc(S_{j_1}^{(t_1)})$ implies $j_1=j_2$, yielding a contradiction to \eqref{non zero homo between two}.
Hence $t<t_1$.
Now we have the following commutative diagram:
$$\xymatrix{
 0\ar[r] &S_{j_2}^{(t_2)} \ar[r]\ar[d]& S_{j_1}^{(t_2+t_1-t)} \ar[r]\ar[d]& S_{j_1}^{(t_1-t)}  \ar@{=}[d] \ar[r]& 0\\
 0\ar[r] & S_{j_2}^{(t)} \ar[r]&S_{j_1}^{(t_1)}\ar[r]& S_{j_1}^{(t_1-t)} \ar[r]& 0 }
$$
which yields a short exact sequence
$$0\to S_{j_2}^{(t_2)}\to S_{j_2}^{(t)}\oplus S_{j_1}^{(t_2+t_1-t)}\to S_{j_1}^{(t_1)}\to 0.$$
Since $\Pi_{\varphi}$ is closed under extensions and direct summands, we see that
$S_{j_2}^{(t)} \in \Pi_{\varphi}$. Since $\soc(S_{j_2}^{(t)})=\soc(S_{j_1}^{(t_1)})$ and $t <t_1$, $\top (S_{j_1}^{(t_1)})$ does not belong to the composition factor set of $S_{j_2}^{(t)}$. Then  by Proposition \ref{non-zero morphism by soc or top between objects} (1),  we see that $\Hom(S_{j_1}^{(t_1)}, S_{j_2}^{(t)})= 0$, a contradiction. Then we are done.
\end{proof}

\begin{lem} \label{at most n-t+1 semistable objects with length t}
For any $1\leq t\leq n$, there are at most $(n-t+1)$-many indecomposable semistable objects of length $t$ in {\rm{\textbf{T}}}$_n$.
\end{lem}

\begin{proof}
Note that any simple object $S_{j}^{} \;(0\leq j\leq n-1)$ is stable since it has no non-trivial subobjects. Then the result holds for $t=1$.
For contradiction we assume there exists $2\leq t\leq n$, such that the number of indecomposable semistable objects of length $t$ is greater than $n-t+1$. Assume they are given by $S_{j_i}^{(t)}$ ($1\leq j\leq n-t+k$) for some $k\geq 2$, where $0\leq j_1 <j_2< \cdots <j_{n-t+k}\leq n-1$.
Clearly, $1\leq j_{i+1}-j_{i} \leq t-1$ for each $i$ since $(t-1)+(n-t+k)>n$.
It follows that $\Hom(S_{j_i}^{(t)}, S_{j_{i+1}}^{(t)})\neq 0$ since $\top(S_{j_{i}}^{(t)})$ belongs to the composition factors set of $S_{j_{i+1}}^{(t)}$.
Since $j_{n-t+k} <j_{1}+n $, we have $\Hom(S_{j_{n-t+k}}^{(t)}, S_{j_{1}}^{(t)})\neq 0$ by similar arguments as above. Then there is a chain of non-zero morphisms $ S_{j_{1}}^{(t)}\rightarrow S_{j_{2}}^{(t)}\rightarrow \cdots \rightarrow S_{j_{n-t+k}}^{(t)} \rightarrow S_{j_{1}}^{(t)}$, which implies $\bm \phi(S_{j_{1}}^{(t)})=\bm \phi(S_{j_{2}}^{(t)})=\cdots=\bm \phi(S_{j_{n-t+k}}^{(t)})$,
 a contradiction to Lemma \ref{different semistable subcategories}. We are done.
\end{proof}

\begin{lem}\label{n height semistable objects in tube}
There exists a unique integer $0\leq j\leq n-1$, such that $S_{j}^{(n)}$ is semistable.
\end{lem}

\begin{proof}
First we assume that each $S_{j}^{(n)}$ is not semistable for $0\leq j\leq n-1$. Note that any simple object is semistable. Then
there exist $0\leq j\leq n-1$ and $1\leq t<n$, such that $S_{j}^{(t)}$ is semistable, but $S_{i}^{(k)}$ is not semistable for any $t<k\leq n$ and $0\leq i\leq n-1$. Consider the following exact sequences
$$ 0\rightarrow S_{j-t}^{(n-t)} \rightarrow S_{j}^{(n)} \rightarrow S_{j}^{(t)}\rightarrow0,\quad\text{and}\quad
0\rightarrow S_{j}^{(t)} \rightarrow S_{j-t}^{(n)} \rightarrow S_{j-t}^{(n-t)}\rightarrow0.$$
Since  $S_{j}^{(t)}$ is the minimal semistable quotient object of $S_{j}^{(n)}$, we get $\bm \phi(S_{j}^{(t)}) <\bm \phi^{-}(S_{j-t}^{(n-t)})$. Similarly, since $S_{j}^{(t)}$ is the maximal semistable subobject of $S_{j-t}^{(n)}$, we get $\bm \phi^{+}(S_{j-t}^{(n-t)}) <\bm \phi(S_{j}^{(t)})$. Hence $\bm \phi(S_{j}^{(t)}) < \bm \phi^{-}(S_{j-t}^{(n-t)}) \leq \bm \phi^{+}(S_{j-t}^{(n-t)}) <\bm \phi(S_{j}^{(t)})$, a contradiction.
Therefore, there exists at least one $S_{j}^{(n)}$ which is semistable.
The uniqueness of $S_{j}^{(n)}$ follows from Lemma \ref{at most n-t+1 semistable objects with length t}. We are done.
\end{proof}

\begin{lem}\label{non semistable objects} For any $0\leq i\leq n-1$, $1\leq s<n$ and $r\geq 1$, $S_{i}^{(rn+s)}$ is not semistable.
\end{lem}

\begin{proof}
Assume $S_{i}^{(rn+s)}$ is semistable with phase $\varphi$ for some $0\leq i\leq n-1$, $1\leq s<n$ and $r\geq 1$. Using Lemma \ref{exact seqs in tube}, one can obtain the following short exact sequence
$$0\to S_{i}^{(rn+s)}\to S_{i}^{(s)}\oplus S_{i}^{(2rn+s)}\to S_{i}^{(rn+s)}\to 0.$$
Since $\Pi_{\varphi}$ is closed under extensions and direct summands, we see that
$S_{i}^{(s)} \in \Pi_{\varphi}$.

On the other hand, by Lemma \ref{n height semistable objects in tube}, there exists a unique $0\leq j\leq n-1$, such that $S_{j}^{(n)}$ is semistable. By Corollary \ref{non-zero morphism between heigh objects}, we know that $\Hom(S_{j}^{(n)}, S_{i}^{(rn+s)})\neq 0\neq \Hom(S_{i}^{(rn+s)}, S_{j}^{(n)})$. Then  we obtain that $\bm \phi(S_{j}^{(n)})=\bm \phi(S_{i}^{(rn+s)})=\varphi=\bm \phi(S_{i}^{(s)})$,
 a contradiction to Lemma \ref{different semistable subcategories}. We are done.
\end{proof}

Denote by $f(t)$ the number of indecomposable semistable objects of length $t$ in {\rm{\textbf{T}}}$_n$. Combining with Lemmas \ref{at most n-t+1 semistable objects with length t}, \ref{n height semistable objects in tube} and \ref{non semistable objects}, we have:
\begin{equation*}
f(t)=
\begin{cases}
n, & \text{if $t=1;$}\\
1,& \text{if $n|t;$}\\
\leq n-t+1, & \text{if $2\leq t< n;$}\\
0, & \text{if else.}\\
\end{cases}
\end{equation*}
Moreover, if $\bm \phi(S_{0}^{})<\bm \phi(S_{1}^{})<\cdots<\bm \phi(S_{n-1}^{})$, then $f(t)=n-t+1$ holds for any $1\leq t\leq n$.

\begin{prop}
Each semistable subcategory $\Pi_{\varphi}$ has the form $\langle S_{j}^{(s)} \rangle$,  where $0\leq j\leq n-1$ and $1\leq s\leq n$.
\end{prop}

\begin{proof}
By Lemma \ref{non semistable objects} we know that any semistable object in {\rm{\textbf{T}}$_n$} has the possible form $S_{j}^{(t)}$ for some $0\leq j\leq n-1$, where $1\leq t< n$ or $n|t$.
According to the proof of Lemma \ref{finitely many extension closed subcategory}, we obtain that $\langle S_{j}^{(rn)}\rangle=\langle S_{j}^{(n)}\rangle$ for any $r\geq 1$. Moreover, by Lemma \ref{n height semistable objects in tube}, there exists a unique indecomposable semistable object of length $n$.
Then by  Lemma \ref{different semistable subcategories}, any two distinct semistable objects $S_{j_1}^{(t_1)}$ and $S_{j_2}^{(t_2)}$ with $t_i\leq n\; (i=1,2)$ have different phases.
We are done.
\end{proof}

Using the above results, we can classify all the finest stability data on any tube category via combinatorial method. In the following we state the classification result for the tube category {\rm{\textbf{T}}$_3$} as an example.

\begin{exam}
Up to $\tau$-actions, there are four equivalent classes of finest stability data $(\Phi,\{\Pi_{\varphi}\}_{\varphi\in\Phi})$ on {\rm{\textbf{T}}$_3$} as follows:

\begin{itemize}
\item[(1)] $\Phi =\{ 1, 2, 3, 4, 5 \}$;

\begin{table}[h]
\begin{tabular}{|c|c|c|c|c|}
\hline
\makecell*[c]{$\Pi_{1}$}&$\Pi_{2}$&$\Pi_{3}$&$\Pi_{4}$&$\Pi_{5}$\\
\hline
\makecell*[c]{$\langle S_0\rangle$}&$\langle S^{}_{2}\rangle$&$\langle S^{(3)}_{1}\rangle$&$\langle  S^{(2)}_{1}\rangle$&$\langle S_1\rangle$\\
\hline
\makecell*[c]{$\langle S_0\rangle$}&$\langle S^{(2)}_{1}\rangle$&$\langle S^{(3)}_{2}\rangle$&$\langle  S^{}_{2}\rangle$&$\langle S_1\rangle$\\
\hline
\end{tabular}
\label{I5}
\end{table}

\newpage

\item[(2)] $\Phi =\{ 1, 2, 3, 4, 5, 6\}$;

\begin{table}[h]
\begin{tabular}{|c|c|c|c|c|c|}
\hline
\makecell*[c]{$\Pi_{1}$}&$\Pi_{2}$&$\Pi_{3}$&$\Pi_{4}$&$\Pi_{5}$&$\Pi_{6}$\\
\hline
\makecell*[c]{$\langle S_0\rangle$}&$\langle S^{(2)}_{1}\rangle$&$\langle S_1\rangle$&$\langle S^{(3)}_{2}\rangle$&$\langle S^{(2)}_{2}\rangle$&$\langle S_2\rangle$\\
\hline
\makecell*[c]{$\langle S_0\rangle$}&$\langle S^{(2)}_{1}\rangle$&$\langle S^{(3)}_{2}\rangle$&$\langle S_1\rangle$&$\langle S^{(2)}_{2}\rangle$&$\langle S_2\rangle$\\
\hline
\end{tabular}
\label{I6}
\end{table}
\end{itemize}
\end{exam}

\subsection{Torsion pairs in tube categories}
We have already obtained all possible semistable subcategories for  finest stability data on the tube category \rm{\textbf{T}}$_n$. This
enables us to classify torsion pairs in \rm{\textbf{T}}$_n$.
\begin{thm}\label{classification of torsion pairs in a tube}
$(\cal{T}, \cal{F})$ is a torsion pair in {\rm{\textbf{T}}$_n$} if and only if one of the following holds (up to $\tau$-actions):
\begin{itemize}
  \item [(1)] $\cal{T}$ is a torsion class in $\langle S_1,S_2, \cdots, S_{n-1}\rangle$, and $\cal{F}=\cal{T}^{\perp}$;
   \item [(2)] $\cal{F}$ is a torsionfree class in $\langle S_0,S_1, \cdots, S_{n-2}\rangle$, and $\cal{T}={}^{\perp}\cal{F}$.
\end{itemize}
\end{thm}

\begin{proof}
If $\cal{T}$ is a torsion class in $\langle S_1,S_2, \cdots, S_{n-1}\rangle$, then $\cal{T}$ is
closed under extensions and quotient objects in $\langle S_1,S_2, \cdots, S_{n-1}\rangle$ and also in \rm{\textbf{T}}$_n$. Hence $\cal{T}$ is a  torsion class in \rm{\textbf{T}}$_n$ by
\cite[Lem. 1.1.3]{Po}. Then $(\cal{T}, \cal{T}^{\perp})$ is a torsion pair in \rm{\textbf{T}}$_n$.

Similarly, if $\cal{F}$ is a torsionfree class in $\langle S_0,S_1, \cdots, S_{n-2}\rangle$, then $\cal{F}$ is
closed under extensions and subobjects in $\langle S_0,S_1, \cdots, S_{n-2}\rangle$ and also in \rm{\textbf{T}}$_n$. Hence $\cal{F}$ is a torsionfree class in \rm{\textbf{T}}$_n$ by the dual of \cite[Lem. 1.1.3]{Po}. Then $({}^{\perp}\cal{F}, \cal{F})$ is a torsion pair in \rm{\textbf{T}}$_n$.

On the other hand, for any torsion torsion pair $(\cal{T}, \cal{F})$ in \rm{\textbf{T}}$_n$, by Propositions \ref{torsion pair construction} and \ref{finest stab data for cyclic quiver}, there exists a finest stability data $(\Phi,\{\Pi_{\varphi}\}_{\varphi\in\Phi})$ and a decomposition $\Phi=\Phi_-\cup\Phi_{+}$, such that $$\cal{T}=\langle\Pi_{\varphi}\mid\varphi\in\Phi_+ \rangle,\quad\cal{F}=\langle\Pi_{\varphi}\mid\varphi\in\Phi_-\rangle.$$
By Lemma \ref{n height semistable objects in tube}, we can assume, up to $\tau$-actions on \rm{\textbf{T}}$_n$, $S_{n-1}^{(n)}$ is the unique semistable object under $(\Phi,\{\Pi_{\varphi}\}_{\varphi\in\Phi})$. Then $S_{n-1}^{(n)}\in\mathcal{F}$ or $S_{n-1}^{(n)}\in\mathcal{T}$. We consider the following two cases.

(1) If $S_{n-1}^{(n)}\in\mathcal{F}$, then $$\cal{T}={}^{\perp}\cal{F}\subseteq {}^{\perp}(S _{n-1}^{(n)})=\langle S_1,S_2, \cdots, S_{n-1}\rangle.$$
      Since $\cal{T}$ is a torsion class in \rm{\textbf{T}}$_n$, it is closed under extensions and quotient objects. The same statements holds in the subcategory $\langle S_1,S_2, \cdots, S_{n-1}\rangle$. Hence $\cal{T}$ is a torsion class in $\langle S_1,S_2, \cdots, S_{n-1}\rangle$. Obviously, by definition we have $\cal{F}=\cal{T}^{\perp}$.

(2) If $S_{n-1}^{(n)}\in\mathcal{T}$, then $$\cal{F}=\cal{T}^{\perp}\subseteq (S_{n-1}^{(n)})^{\perp}=\langle S_0,S_1, \cdots, S_{n-2}\rangle.$$
      Since $\cal{F}$ is a torsionfree class in \rm{\textbf{T}}$_n$, it is closed under extensions and subobjects. The same statements holds in the subcategory $\langle S_0,S_1, \cdots, S_{n-2}\rangle$. Hence $\cal{F}$ is a torsionfree class in $\langle S_0,S_1, \cdots, S_{n-2}\rangle$. It follows that $\cal{T}={}^{\perp}\cal{F}$ and we are done.
\end{proof}

\begin{rem} In \cite{BBM}, Baur-Buan-Marsh classified the torsion pairs in \rm{\textbf{T}}$_n$ via maximal rigid objects in the extension \textbf{$\overline{\textbf{T}}$}$_{n}$ of the tube category \rm{\textbf{T}}$_n$. Here, \textbf{$\overline{\textbf{T}}$}$_{n}$ is the subcategory of Mod-$\textbf{k} C_n$ for cyclic quiver $C_n$, whose objects are all filtered direct limits or
filtered inverse limits of objects in \rm{\textbf{T}}$_n$.

 More precisely, there are two different types of torsion pairs: the ray type and the coray type, where the ray type corresponds to the maximal rigid object of Pr\"ufer type in \textbf{$\overline{\textbf{T}}$}$_{n}$; and the coray type corresponds to the maximal rigid object of adic type in \textbf{$\overline{\textbf{T}}$}$_{n}$. In fact, up to $\tau$-actions, the torsion pairs in \textbf{$\overline{\textbf{T}}$}$_{n}$ of ray type (\emph{resp.} coray type) correspond to those in (1) (\emph{resp.} (2)) in Theorem \ref{classification of torsion pairs in a tube}.
\end{rem}
As an explanation of Theorem \ref{classification of torsion pairs in a tube}, we list all the torsion pairs in the tube category \rm{\textbf{T}}$_3$ as below.

\begin{exam}
Up to $\tau$-actions,  the non-trivial torsion pairs $(\cal{T}, \cal{F})$ in \rm{\textbf{T}}$_3$  are classified as below:

\begin{table}[h]
\caption{Non-trivial torsion pairs $(\cal{T}, \cal{F})$ in \rm{\textbf{T}}$_3$}
\begin{tabular}{|c|c|c|c|c|c|}
\hline
\multicolumn{2}{|c|}{Ray type}&&\multicolumn{2}{|c|}{Coray type}\\
\cline{1-2} \cline{4-5}
\makecell*[c]{$\cal{T}$}&$\cal{F}$&&$\cal{T}$&$\cal{F}$\\
\cline{1-2} \cline{4-5}
\makecell*[c]{$\langle S_2 \rangle$}&$\langle S^{}_{0},  S^{}_{1}, S^{(2)}_{2} \rangle$&&$\langle  S^{}_{1},  S^{}_{2}, S^{(2)}_{1} \rangle$&$\langle S_0 \rangle$ \\
\cline{1-2} \cline{4-5}
\makecell*[c]{$\langle S_2, S^{(2)}_{2}\rangle$}&$\langle S^{}_{0},  S^{}_{1}, S^{(3)}_{2}\rangle$&&$\langle S^{}_{1}, S^{}_{2},  S^{(3)}_{2}\rangle$&$\langle S_0, S^{(2)}_{1}\rangle$\\
\cline{1-2} \cline{4-5}
\makecell*[c]{$\langle S_1, S^{}_{2}\rangle$}&$\langle S^{}_{0},  S^{(2)}_{1}, S^{(3)}_{2} \rangle$&&$\langle  S^{}_{2}, S^{(2)}_{2}, S^{(3)}_{2} \rangle$&$\langle S_0, S^{}_{1} \rangle$\\
\hline
\end{tabular}
\label{table1}
\end{table}

\end{exam}

\section{Tame hereditary algebra}

In this section we investigate the stability data for tame hereditary algebras. We show that each stability data can be refined to a finest one for the category of finitely generated modules  over a tame hereditary algebra.

We recall from \cite{Ringel} and \cite{SimSkow} for the definition and the main properties for tame hereditary algebra $\Lambda=\textbf{k}Q$, where $Q$ is a tame quiver. The category  $\mod$-$\Lambda$ of finitely generated modules over $\Lambda$ is a hereditary abelian category.
Moreover, there is a decomposition of indecomposable $\Lambda$-modules: $\ind (\mod$-$\Lambda)=\cal{P}\vee\cal{R}\vee\cal{I}$,  where $\cal{P}$ is the postprojective component, $\cal{R}$ is the regular component, and $\cal{I}$ is the preinjective component.
The postprojective component $\cal{P}$ consists of modules $\tau^{-n}P_{i}, i\in Q_0, n\in \mathbb{Z}_{\geq 0}$, where $P_{i}$ are indecomposable projective modules, and each indecomposable module $M$ in $\cal{P}$ is directing with $\End(M) = \textbf{k}, \Ext^{i} (M, M) = 0$ for all $i \geq 1$; the regular component $\cal{R}$ consists of a family ${\textbf{T}}=\{ \Gamma({\textbf{T}_{{\textbf{\textit{x}}}}) \}_{{\textbf{\textit{x}}} \in \mathbb{P}^{1} }  }$ of pairwise orthogonal stable tubes $\Gamma(\textbf{T}_{{\textbf{\textit{x}}}})$, in particular, all but finitely many of the tubes $\Gamma(\textbf{T}_{{\textbf{\textit{x}}}})$ in $\cal{R}$ are homogeneous, and there are at most $|Q_0|-2$ non-homogeneous tubes;
and the preinjective component $\cal{I}$ consists of modules $\tau^{m}I_{i}, i\in Q_0, m\in \mathbb{Z}_{\geq 0}$, where $I_{i}$ are indecomposable injective modules, and
each indecomposable molule $N$ in $\cal{I}$ is also directing with $\End(N) = \textbf{k},  \Ext^{i} (N,N) = 0$ for all $i \geq 1$.

In the rest of this section we denote by $\cal{A}=\mod$-$\Lambda$, where $\Lambda$ is a tame hereditary algebra.

\begin{prop}\label{finest for tame hereditary alg} Each stability data on $\cal{A}=\mod$-$\Lambda$ can be refined to a finest one.
\end{prop}

\begin{proof}
Let $(\Phi,\{\Pi_{\varphi}\}_{\varphi\in\Phi})$ be a stability data on $\cal{A}$. Since each indecomposable module $M$ in the postprojective component $\cal{P}$  is directing, the indecomposable modules in the subcategory $\Pi_{\varphi}\cap \cal{P}$ can be  ordered by a linearly ordered set $\cal{P}_{\varphi}$ for any $\varphi\in\Phi$, namely, $\Pi_{\varphi}\cap \cal{P}=\{M_{\psi} \ |\ \psi\in \cal{P}_{\varphi}\}$, such that $\Hom(M_{\psi'}, M_{\psi''})=0$ for any $\psi'>\psi''\in \cal{P}_{\varphi}$.

Similarly, since each indecomposable module $N$ in the preinjective component $\cal{I}$  is directing, the indecomposable modules in the subcategory  $\Pi_{\varphi}\cap \cal{I}$ can be ordered by a linearly ordered set $\cal{I}_{\varphi}$ for any $\varphi\in\Phi$, namely, $\Pi_{\varphi}\cap \cal{I}=\{N_{\psi}\ |\ \psi\in \cal{I}_{\varphi}\}$, such that $\Hom(N_{\psi'}, N_{\psi''})=0$ for any $\psi'>\psi''\in \cal{I_{\varphi}}$.
Meanwhile, the subcategory $\Pi_{\varphi}\cap\cal{R}$ of regular component $\cal{R}$ can be ordered by a linearly ordered set $\cal{R}_{\varphi}$ for any $\varphi\in\Phi$, namely,
$\Pi_{\varphi}\cap\cal{R}=\{\Pi_{\varphi}\cap
\Gamma(\textbf{T}_{{\textbf{\textit{x}}}})\ |\ (\varphi, {\textbf{\textit{x}}}) \in \cal{R}_{\varphi}, \Pi_{\varphi}\cap\Gamma(\textbf{T}_{{\textbf{\textit{x}}}})\neq 0\}$.

Put $I_{\varphi}=\cal{P}_{\varphi}\cup\cal{R}_{\varphi}\cup\cal{I}_{\varphi}$. Then $I_{\varphi}$ is a linearly ordered set with further relations $\psi'>\psi''>\psi'''$ for any $\psi'\in \cal{I}_{\varphi}, \psi''\in \cal{R}_{\varphi}$ and $\psi'''\in \cal{P}_{\varphi}$.
Let $\Psi= \cup_{\varphi \in \Phi } I_{\varphi} $, then $\Psi$ is a linearly ordered set which keeps each $I_{\varphi}$ as ordered subsets, and if
$\psi'\in I_{\varphi'}, \psi''\in I_{\varphi''}$ with $\varphi'>\varphi''$ in $\Phi$, then $\psi'>\psi''$.

 For any $\varphi\in\Phi$ and $\psi\in I_{\varphi}$, we define
 \begin{equation*}
P_{\psi}=
\begin{cases}
\add~M_{\psi}, & \text{if $\psi\in\cal{P}_{\varphi};$}\\
\add~N_{\psi},& \text{if $\psi\in\cal{I}_{\varphi};$}\\
\add~\{\Pi_{\varphi}\cap
 \Gamma(\textbf{T}_{{\textbf{\textit{x}}}})\}, & \text{if $\psi=(\varphi, {\textbf{\textit{x}}}) \in \cal{R}_{\varphi}$.}\\
\end{cases}
\end{equation*}
By construction, $P_{\psi}$ is extension closed, $\Hom(P_{\psi'}, P_{\psi''})=0$ for any $\psi'>\psi'' \in \Psi$,
and $\Pi_{\varphi}=\cup_{\psi \in I_{\varphi}} P_{\psi}$. It is easy to see that $(\Psi, \{P_{\psi}\}_{\psi \in \Psi})$ is a stability data on $\cal{A}$, which is finer than $(\Phi,\{\Pi_{\varphi}\}_{\varphi\in\Phi})$. Note that each subcategory $P_{\psi}$ for any $\psi$ in $\cal{P}_{\varphi}$ or $\cal{I}_{\varphi}$ contains a unique indecomposable direct summand.
Moreover, for any $\psi=(\varphi, {\textbf{\textit{x}}})\in\cal{R}_{\varphi}$, the semistable subcategory $P_{\psi}=\add~\{\Pi_{\varphi}\cap
\Gamma(\textbf{T}_{{\textbf{\textit{x}}}})\}$ is a subcategory of a tube category $\textbf{T}_{{\textbf{\textit{x}}}}$. Then by similar arguments as in the proof of Proposition  \ref{finest stab data for cyclic quiver},
we can make local refinement to obtain a finest stability data. This finishes the proof.
\end{proof}

Consequently, we can classify torsion pairs for the category of finitely generated modules over any tame heredirary algebra via the stability data approach. In the following we provide a concrete example for tame heredirary algebra $\widetilde{A}_1$.

 \begin{exam}
Let $Q:1 \rightrightarrows 2$. Then $\widetilde{A}_1$:=$\mathbf{k}Q$ is the minimal tame hereditary algebra. Denote by $P_i, I_i, S_i\; (1\leq i\leq 2)$ the indecomposable projective, injective and simple $\widetilde{A}_1$-modules respectively.
The postprojective component $\cal{P}$ consists of modules $\tau^{-n}P_{i}$ and the preinjective component $\cal{I}$ consists of modules $\tau^{n}I_{i}$, where $i=1,2$ and $n\in \mathbb{Z}_{\geq 0}$. For convenience, we denote by $\cal{P}_{2k+1}{:=} \tau^{-k}P_{2}$, $\cal{P}_{2k+2}:= \tau^{-k}P_{1}$, and $\cal{I}_{2k+1}{:=} \tau^{k}I_{1}$, $\cal{I}_{2k+2}:= \tau^{k}I_{2}$, for any $k\in\mathbb{Z}_{\geq 0}$. The regular component $\cal{R}$ consists of modules $S^{(d)}_{\textbf{\textit{x}}}, {\textbf{\textit{x}}} \in \mathbb{P}^{1}, d\in \mathbb{Z}_{\geq 1}$ where each  $S_{{\textbf{\textit{x}}}}$ is a simple regular module.

For any finest stability data $(\Phi,\{\Pi_{\varphi}\}_{\varphi\in\Phi})$ on $\mod$-$\widetilde{A}_1$, the simple modules $S_1$ and $S_2$ are stable with $\bm \phi(S_{1})\neq \bm \phi(S_{2})$ by Theorem \ref{iff cond for finest}.
In fact, there are only two equivalent classes of finest stability data on $\mod$-$\widetilde{A}_1$ as below, according to $\bm \phi(S_{1})>\bm \phi(S_{2})$ or $\bm \phi(S_{1})< \bm \phi(S_{2})$ respectively.

\begin{itemize}
 \item [(1)]  $\Phi =(0, \mathbb{Z}_{\geq 1})\cup\mathbb{P}^{1}\cup(1, \mathbb{Z}_{\geq 1})$ with $(0, n)<(0, n+1)<{\textbf{\textit{x}}}<(1, n+1)<(1, n)$ for any $n\in\mathbb{Z}_{\geq1}, {\textbf{\textit{x}}} \in \mathbb{P}^{1}$, and there are no restrictions on the choice of ordering on the set $\mathbb{P}^{1}$. For any $k\in\mathbb{Z}_{\geq1}$, $\Pi_{(0, k)}=\langle \cal{P}_{k}\rangle$ and $\Pi_{(1, k)}=\langle \cal{I}_{k}\rangle$; for any ${\textbf{\textit{x}}}\in \mathbb{P}^{1}$,  $\Pi_{{\textbf{\textit{x}}}}=\langle S_{\textbf{\textit{x}}}\rangle$. In this case, each indecomposable module is semistable;

\item [(2)] $\Phi =\{ 1, 2  \}$ with $1<2$, and $\Pi_{ 1}=\langle S_1\rangle,  \Pi_{ 2}=\langle S_2\rangle$.
In this case, only $S_1, S_2$ are semistable indecomposable modules. For any other indecomposable module $X$, assume $X$ has dimension vector $\underline{\rm{\textbf{dim}}}\; X =(m,  n)$, its HN-filtration is given by
$$0\rightarrow S^{\oplus n}_2 \rightarrow X \rightarrow S^{\oplus m}_1 \rightarrow 0.$$
\end{itemize}

As a consequence of Propositions \ref{finest for tame hereditary alg} and \ref{torsion pair construction}, we can obtain a classification of torsion pairs in $\mod$-$\widetilde{A}_1$ as below, 
where $P\subseteq \mathbb{P}^{1}$ is a (possibly empty) subset of $\mathbb{P}^{1}$ and $n\in\mathbb{Z}_{\geq1} $.

\begin{table}[h]
\caption{ Non-trivial torsion pairs $(\cal{T},\cal{F})$ in $\mod$-$\widetilde{A}_1$}
\begin{tabular}{|c|c|c|c|c|}
  \hline
 \makecell*[c]{$\cal{T}$}&$\cal{F}$\\
\hline
 \makecell*[c]{$\langle  S_{{\textbf{\textit{x}}}}, \cal{I}  \mid {\textbf{\textit{x}}} \in P \rangle  $}&$ \langle\cal{P},   S_{{\textbf{\textit{x}}}} \mid    {\textbf{\textit{x}}}\in\mathbb{P}^{1}\backslash P\rangle$ \\
  \hline
 \makecell*[c]{$\langle \cal{I}_{m} \ |\ m\leq n \rangle$}&$ \langle\cal{P}, \cal{R}, \cal{I}_{m} \ |\ m> n\rangle$\\
  \hline
   \makecell*[c]{$\langle \cal{P}_{m},  \cal{R}, \cal{I} \ |\ m> n \rangle$}&$  \langle  \cal{P}_{m} \ |\ m\leq n \rangle$\\
  \hline
   \makecell*[c]{$\langle  S_{2} \rangle$}&$ \langle S_{1} \rangle $\\
  \hline
 \end{tabular}
 \label{table 2}
\end{table}

Recall that a subcategory $\cal{T}$  is called \emph{contravariantly finite} in an abelian category $\cal{A}$ if for each  $X \in \cal{A}$, there is a
map $f: T \to X$ with  $T \in \cal{T}$ such that $\Hom(T', T ) \stackrel{\cdot f}\longrightarrow  \Hom(T',  X )$ is surjective for all $T' \in \cal{T}$. Dually, a \emph{covariantly finite} subcategory is defined. A subcategory $\cal{T}$ is called  \emph{functorially
finite}  if it is both contravariantly finite and covariantly finite. We call a torsion pair $(\cal{T},\cal{F})$  functorially
finite if the torsion class $\cal{T}$ is functorially
finite.

We remark that all the torsion pairs in the above table are functorially finite except for the first case. Moreover, for any functorially finite torsion pair, its torsion class $\mathcal{T}$ is determined by a support $\tau$-tilting module, c.f. \cite[Thm. 2.7]{AIR}. More precisely, the support $\tau$-tilting module $T=S_{i}$ if $\cal{T}=\langle S_{i} \rangle$ for $i=1,2$, and $T=\cal{I}_{n}\oplus \cal{I}_{n-1}$ if $\cal{T}=\langle \cal{I}_{m} \ |\ m\leq n \rangle$ with $ n\geq2$; while
 $T=\cal{P}_{n+1}\oplus \cal{P}_{n+2}$ if $\cal{T}=\langle \cal{P}_{m},  \cal{R}, \cal{I} \ |\ m> n \rangle$.
\end{exam}

\section{Projective line}

In this section we investigate the stability data for the category $\coh \mathbb{P}^{1}$ of coherent sheaves over the projective line $\mathbb{P}^{1}:= \mathbb{P}_{\textbf{k}}^{1}$. We obtain classifications of finest stability data and torsion pairs on $\coh \mathbb{P}^{1}$.

Recall from \cite{Gro57} and \cite{Lenzing} that $\coh \mathbb{P}^{1}$ is a (skeletally) small and Hom-finite hereditary abelian $\textbf{k}$-linear category
with Serre duality of the form
$$\Ext^{1}(X,Y) \cong \DHom(Y, \tau X),$$
where $\D=\Hom_\textbf{k}(-, \textbf{k}) $ and $\tau$ is given by the grading shift with $(-2)$.
The subcategories ${\rm{coh}}_0 \mathbb{P}^{1}$ of torsion sheaves and $\vect \mathbb{P}^{1}$ of vector bundles form a split torsion pair $({\rm{coh}}_0 \mathbb{P}^{1}, \vect \mathbb{P}^{1})$ in $\coh \mathbb{P}^{1}$, namely, any coherent sheaf can be decomposed as a direct sum of a torsion sheaf and a vector bundle, and there are no nonzero homomorphisms from ${\rm{coh}}_0 \mathbb{P}^{1}$ to  $\vect \mathbb{P}^{1}$.
The  subcategory ${\rm{coh}}_0 \mathbb{P}^{1}$  splits into a coproduct of connected subcategories $\coprod _{x\in\mathbb{P}^{1}} \textbf{T}_{{\textbf{\textit{x}}}}$, where each $\textbf{T}_{{\textbf{\textit{x}}}}$ is generated by a simple sheaf $S_{{\textbf{\textit{x}}}}$, and the associated AR-quivers $\Gamma(\textbf{T}_{{\textbf{\textit{x}}}})$ is a stable tube of rank one. Any indecomposable vector bundle in $\vect \mathbb{P}^{1}$ is a line bundle, more precisely, it is of the form $\mathcal{O}(n)$ for $n\in \mathbb{Z}$. If $n=0$, then $\mathcal{O}:=\mathcal{O}(0)$ is the structure sheave of $\mathbb{P}^{1}$. The homomorphism space $\Hom(\mathcal{O}(m), \mathcal{O}(n))$ between two line bundles has dimension $n-m+1$ if $n \geq m$ and $0$ otherwise. Moreover, if $\mathcal{O}(n)$ is a line bundle and $S_{{\textbf{\textit{x}}}}$ is a simple sheaf, then $\Hom(\mathcal{O}(n), S_{{\textbf{\textit{x}}}}) \cong\textbf{k}$.

The following lemma shows that any indecomposable sheaf is semistable for $\coh {\mathbb{P}^1}$.

\begin{lem} \label{semistable objects for P1}
Let $(\Phi, \{\Pi_{\varphi}\}_{\varphi\in\Phi})$ be a stability data on $\coh \mathbb{P}^{1}$, then any indecomposable sheaf is semistable.
\end{lem}

\begin{proof}
Obviously, any simple sheaf $S_{{\textbf{\textit{x}}}}$ in $\coh \mathbb{P}^{1}$ is stable since it has no non-trivial subsheaves. Then $S_{\textbf{\textit{x}}}\in\Pi_{\varphi}$ for some $\varphi$. It follows that the subcategory $\langle S_{{\textbf{\textit{x}}}}\rangle\subseteq \Pi_{\varphi}$, hence any torsion sheaf $S_{{\textbf{\textit{x}}}}^{(n)}\in\langle S_{{\textbf{\textit{x}}}}\rangle$ is semistable.

For any line bundle $L\in\coh \mathbb{P}^{1}$, if it is not semistable, then its minimal semistable quotient sheaf is given by a torsion sheaf $X$, and its maximal semistable subsheaf is given by another line bundle $L'$. Then $\Hom(L', X)\neq 0$ yields a contradiction.
We are done.
\end{proof}

There is a classical slope function $\mu: \ind(\coh  {\mathbb{P}^1})\to \mathbb{Z}\cup \{\infty\}$ on the category of coherent sheaves over $\mathbb{P}^1$, satisfying $\mu(\co(n))=n$ for any line bundle $\co(n)$ and $\mu(X)=\infty$ for any torsion sheaf $X$. For any indecomposable sheaves $X,Y\in\coh  {\mathbb{P}^1}$ with $\mu(X)\neq\mu(Y)$, it is well-known that $$\Hom(X, Y)\neq 0 \quad\text{if \;and \;only \;if}\quad \mu(X)<\mu(Y).$$
Denote by $\overline{\mathbb{Z}}=\mathbb{Z}\cup \{\infty\}$, which is a linearly ordered set in the natrual way. Then the slope function gives a stability data $(\overline{\mathbb{Z}}, \{P_{\psi}\}_{\psi\in\overline{\mathbb{Z}}})$, called \emph{slope stability data}, where $$P_{\psi}=\langle X\in \ind (\coh \mathbb{P}^1)\;|\; \mu(X)=\psi\rangle.$$

The following result shows that each finest stability data on $\coh {\mathbb{P}^1}$ can be obtained from the slope stability data by local refinements.

\begin{prop} \label{refine from slope for P1}
Each finest stability data on $\coh {\mathbb{P}^1}$ is a refinement of the slope stability data $(\overline{\mathbb{Z}}, \{P_{\psi}\}_{\psi\in\overline{\mathbb{Z}}})$.
\end{prop}

\begin{proof}
Let $(\Phi, \{\Pi_{\varphi}\}_{\varphi\in\Phi})$ be a finest stability data on $\coh {\mathbb{P}^1}$.
We only need to show that $(\Phi, \{\Pi_{\varphi}\}_{\varphi\in\Phi})$ is finer than $(\overline{\mathbb{Z}}, \{P_{\psi}\}_{\psi\in\overline{\mathbb{Z}}})$.

For any non-zero indecomposable sheaves $X,Y\in\Pi_{\varphi}$, we have $\Hom(X,Y)\neq 0\neq \Hom(Y, X)$. It follows that $\mu(X)=\mu(Y)$. This induces a well-defined map $r: \Phi\to \overline{\mathbb{Z}}$, satisfying $r(\varphi)=\mu(X)$ for any indecomposable $X\in\Pi_{\varphi}$. Moreover, by Lemma
\ref{semistable objects for P1}, we know that any indecomposable sheaf is semistable under the stability data $(\Phi, \{\Pi_{\varphi}\}_{\varphi\in\Phi})$. It follows that $r$ is surjective.

By definition, it remains to show the statements (1) and (2) in
Definition \ref{finer or coarser} hold.  In fact,
for any  $\varphi_1>\varphi_2\in\Phi$ and any indecomposable $X_i\in\Pi_{\varphi_i}\,(i=1,2)$, we have $\Hom(X_1, X_2)=0$. Hence $\mu(X_1)\geq \mu(X_2)$. That is, $r(\varphi_1)\geq r(\varphi_2)$.
On the other hand, for any $\psi\in\overline{\mathbb{Z}}$, $$P_{\psi}=\langle  X\in\ind~(\coh\mathbb{P}^1)\ |\ \mu(X)=\psi\rangle=\langle  X\ |\ X\in\Pi_{\varphi}, r(\varphi)=\psi\rangle
=\langle  \Pi_{\varphi}\ |\ \varphi\in r^{-1}(\psi)\rangle.$$

Therefore, $(\Phi, \{\Pi_{\varphi}\}_{\varphi\in\Phi})$ is finer than $(\overline{\mathbb{Z}}, \{P_{\psi}\}_{\psi\in\overline{\mathbb{Z}}})$. We are done.
\end{proof}

The following result describes semistable subcategories for finest stability data on $\coh \mathbb{P}^{1}$.

\begin{prop} \label{semistable obj in proj line}
Let $(\Phi,\{\Pi_{\varphi}\}_{\varphi\in\Phi})$ be a finest stability data on $\coh \mathbb{P}^{1}$. Then each semistable subcategory has the form $\langle \mathcal{O}(n) \rangle$ for $ n \in \mathbb{Z}$, or $\langle S_{{\textbf{\textit{x}}}} \rangle$ for some $ {\textbf{\textit{x}}} \in \mathbb{P}^{1}$.
Further, the phases for the stable sheaves satisfy:
\begin{itemize}
  \item[(1)] $\bm \phi(\mathcal{O}(n)) <\bm \phi(\mathcal{O}(n+1)) < \bm \phi(S_{{\textbf{\textit{x}}}}), \forall \ n \in \mathbb{Z}, {\textbf{\textit{x}}} \in \mathbb{P}^{1}$;
  \item[(2)] $\bm \phi(S_{{\textbf{\textit{x}}}})\not=\bm \phi(S_{{\textbf{\textit{y}}}}),  \forall \ {\textbf{\textit{x}}}\neq {\textbf{\textit{y}}}\in \mathbb{P}^{1}$, and there are no restrictions on the choice of ordering on the set $\{\bm \phi(S_{{\textbf{\textit{x}}}}) \ | \ {\textbf{\textit{x}}} \in \mathbb{P}^{1}\}$.
\end{itemize}
Consequently, any stability data on $\coh \mathbb{P}^{1}$ can be refined to a finest one.
\end{prop}

\begin{proof}
By Lemma \ref{semistable objects for P1}, we know that any indecomposable sheaf is semistable in $\coh \mathbb{P}^{1}$. Notice that $\Hom(\mathcal{O}(n),\mathcal{O}(m)) \neq 0$ if and only if $n\leq m$, and $\Hom(\mathcal{O}(n), S_{{\textbf{\textit{x}}}})\neq 0,\ \forall \ n \in \mathbb{Z}, {\textbf{\textit{x}}} \in \mathbb{P}^{1}$. Moreover, there are no non-zero homomorphisms between torsion sheaves concentrated in distinct points. Then the result follows from Theorem \ref{iff cond for finest} and Lemma \ref{semistable objects for P1} immediately.
\end{proof}

Note that $\mathbb{Z}$ is a linearly ordered set in the natural way, i.e., $n<n+1$ for any $n\in \mathbb{Z}$. Meanwhile, $\mathbb{P}^{1}$ can be viewed as a linearly ordered set by fixing any choice of ordering on the points of the projective line. Then the union $\mathbb{Z}\cup \mathbb{P}^{1}$ also forms a linearly ordered set by composing the relations $n<{\textbf{\textit{x}}}$ for any $n\in\mathbb{Z}$ and ${\textbf{\textit{x}}}\in\mathbb{P}^{1}$. Note that there are infinitely many ordered relations on the set $\mathbb{Z}\cup \mathbb{P}^{1}$, depending on the choices of ordering on the set $\mathbb{P}^{1}$.

As a consequence of Proposition \ref{semistable obj in proj line}, we obtain the following classification result for finest stability data on $\coh \mathbb{P}^{1}$, parameterized by the linearly ordered set $\mathbb{Z}\cup \mathbb{P}^{1}$.

\begin{prop} \label{Finest stab data in proj line} Any finest stability data on $\coh \mathbb{P}^{1}$ has the form $$(\mathbb{Z}\cup \mathbb{P}^{1}, \{\Pi_{\varphi}\}_{\varphi \in\mathbb{Z}\cup \mathbb{P}^{1}}),$$ where  $\mathbb{Z}\cup \mathbb{P}^{1}$ is an arbitrary linearly ordered set defined as above,
$\Pi_{n}=\langle \mathcal{O}(n) \rangle$ for any $ n \in \mathbb{Z}$,
and $\Pi_{{\textbf{\textit{x}}}}=\langle S_{{\textbf{\textit{x}}}} \rangle$ for any ${\textbf{\textit{x}}} \in \mathbb{P}^{1}$.
\end{prop}

Note that each stability data on $\coh \mathbb{P}^{1}$ is coarser than a finest one. Then by Proposition \ref{torsion pair construction},  any torsion pair $(\cal{T}, \cal{F})$ can be obtained from the finest stability data.
As a consequence of Propositions \ref{Finest stab data in proj line} and \ref{torsion pair construction}, we can obtain a classification of torsion pairs in $\coh \mathbb{P}^{1}$.
\begin{prop}
 The non-trivial torsion pairs $(\cal{T},\cal{F})$ in $\coh \mathbb{P}^{1}$ are classified as below, where $P$ is a non-empty subset of $\mathbb{P}^{1}$ and $n$ is an integer.
\end{prop}

\begin{table}[h]
\caption{Non-trivial torsion pairs $(\cal{T},\cal{F})$ in $\coh \mathbb{P}^{1}$}
\begin{tabular}{|c|c|c|c|c|}
  \hline
 \makecell*[c]{$\cal{T}$}&$\cal{F}$\\
\hline
 \makecell*[c]{$\langle S_{{\textbf{\textit{x}}}} \mid {\textbf{\textit{x}}} \in P \subseteq \mathbb{P}^{1}\rangle$}&$\langle  \mathcal{O}(m), S_{{\textbf{\textit{x}}}} \mid   m \in \mathbb{Z},  {\textbf{\textit{x}}}\in \mathbb{P}^{1}\backslash P\rangle$ \\
  \hline
 \makecell*[c]{$\langle  \mathcal{O}(m), S_{{\textbf{\textit{x}}}} \mid m > n, {\textbf{\textit{x}}} \in \mathbb{P}^{1}  \rangle$}&$\langle \mathcal{O}(m) \mid m \leq n \rangle$\\
  \hline
 \end{tabular}
 \label{table 2}
\end{table}

 We remark that the torsion pair of the first case is not functorially finite, while it is functorially finite for the second case, and it is determined by the canonical tilting sheaf $\co(n)\oplus \co(n+1)$.

\section{Elliptic curves}

In this section we investigate the stability data for the category $\coh \mathbb{T}$ of coherent sheaves over a smooth
elliptic curve $\mathbb{T}$. We will classify finest stability data and torsion pairs for $\coh \mathbb{T}$.

We recall from \cite{Atiy,Lenzing,LM} for the main properties of the category $\coh \mathbb{T}$ as below. Similar as $\coh \mathbb{P}^{1}$, the category $\coh \mathbb{T}$ is a (skeletally) small and Hom-finite hereditary abelian  $\textbf{k}$-linear category, satisfying the $1$-Calabi-Yau property
$$\Ext^{1}(X,Y) \cong \DHom(Y, X),$$
where $\D=\Hom_\textbf{k}(-, \textbf{k}) $.
The subcategory ${\rm{coh}}_0 \mathbb{T}$ of torsion sheaves splits into a coproduct of connected subcategories $\coprod _{{\textbf{\textit{x}}}\in\mathbb{T}} \textbf{T}_{{\textbf{\textit{x}}}}$,  whose associated AR-quivers $\Gamma(\textbf{T}_{{\textbf{\textit{x}}}})$ are stable tubes of rank one generated by the simple sheaves $S_{{\textbf{\textit{x}}}}$. The subcategory $\cal{A}_{q}\; (q\in\mathbb{Q})$ generated by all indecomposable vector bundles of slope $q$ splits into a coproduct of connected tube subcategories $\coprod _{{\textbf{\textit{x}}}\in\mathbb{T}} \cal{A}_{q,{\textbf{\textit{x}}}}$, whose associated AR-quivers $\Gamma(\cal{A}_{q,{\textbf{\textit{x}}}})$ are stable tubes of rank one generated by the unique quasi-simple sheaves $S_{q,{\textbf{\textit{x}}}}$.
It turns out that $\coh\mathbb{T}=\bigvee_{q \in \mathbb{Q} \cup \{\infty\}} \cal{A}_{q}$, where $\cal{A}_{\infty}:={\rm{coh}}_0 \mathbb{T}$. There exists a nonzero morphism from $\cal{A}_{q}$ to $\cal{A}_{r}$ if and only if $q\leq r$. Moreover, if $L$ is a line bundle and $S_{{\textbf{\textit{x}}}}$ is a simple sheaf, then $\Hom(L, S) \cong\textbf{k}$.

The major difference between $\coh \mathbb{T}$ and $\coh \mathbb{P}^{1}$ are:

(1)\ the Grothendieck group K$_{0}(\mathbb{T})$ has infinite rank while  K$_{0}(\mathbb{P}^{1})$ is free abelian of finite rank, thus permitting a close investigation of the Euler form, roots and radical vectors;

(2)\  all simple ordinary sheaves on $\mathbb{P}^{1}$ do have the same class in K$_{0}(\mathbb{P}^{1})$, while simple sheaves concentrated at different points from $\mathbb{T}$ possess distinct classes in  K$_{0}(\mathbb{T})$;

(3)\  the line bundles of degree zero on $\mathbb{T}$ form a one-parameter family, indexed by $\mathbb{T}$, while $\mathbb{P}^{1}$ has a unique line bundle of degree zero (i.e., the structure sheaf);

(4)\  there are no exceptional sheaves on $\mathbb{T}$, while the study of exceptional sheaves on $\mathbb{P}^{1}$ forms a major item.

The following lemma shows that any indecomposable sheaf is semistable for $\coh \mathbb{T}$.

\begin{lem} \label{semistable objects for T}
Let $(\Phi, \{\Pi_{\varphi}\}_{\varphi\in\Phi})$ be a stability data on $\coh \mathbb{T}$, then any indecomposable sheaf is semistable.
\end{lem}

\begin{proof}
Obviously, any simple sheaf $S_{{\textbf{\textit{x}}}}$ in $\coh \mathbb{T}$ is stable since it has no non-trivial subsheaves. Then $S_{\textbf{\textit{x}}}\in\Pi_{\varphi}$ for some $\varphi$. It follows that the subcategory $\langle S_{{\textbf{\textit{x}}}}\rangle\subseteq \Pi_{\varphi}$, hence any torsion sheaf $S_{{\textbf{\textit{x}}}}^{(n)}\in\langle S_{{\textbf{\textit{x}}}}\rangle$ is semistable.

For any quasi-simple vector bundle $E\in\coh \mathbb{T}$ of slope $q$, if it is not semistable, then its maximal semistable subsheaf is given by a sheaf $E_1$ of slope $q_1 < q$, and its minimal semistable quotient sheaf is given by another sheaf $E_2$ of slope $q_2>q$. Then $\Hom(E_1, E_2)\neq 0$ yields a contradiction. Hence any quasi-simple vector bundle is stable. In general, for any indecomposable vector bundle $F$, we know that the top of $F$ (=$\mathrm{top} (F)$) is quasi-simple, and $F$ belongs to the subcategory $\langle \mathrm{top} (F)\rangle$, whose AR-quiver is a homogeneous tube. Hence $\mathrm{top} (F)$ is semistable implies $F$ is semistable.
We are done.
\end{proof}

There is a classical slope function $\mu: \ind (\coh  {\mathbb{T}})\to \mathbb{Q}\cup \{\infty\}$ on the category of indecomposable coherent sheaves over $\mathbb{T}$, satisfying $\mu(E)=\deg(E) / \rank(E)$ for any sheaf $E$, c.f. \cite[Part 1]{Atiy}. Then $\mu(E)\in \mathbb{Q}$ when $E$ is a vector bundle, and $\mu(X)=\infty$ for any torsion sheaf $X$.
Denote by $\overline{\mathbb{Q}}=\mathbb{Q}\cup \{\infty\}$, which is a  linearly ordered set in the natrual way. Then the slope function gives a stability data $(\overline{\mathbb{Q}}, \{P_{\psi}\}_{\psi\in\overline{\mathbb{Q}}})$, called \emph{slope stability data}, where $$P_{\psi}=\langle X\in \ind (\coh \mathbb{T})\;|\; \mu(X)=\psi\rangle.$$

By using the same proof as in Proposition \ref{refine from slope  for P1}, one obtains that
\begin{prop} \label{refine from slope  for T}
Each finest stability data on $\coh {\mathbb{T}}$ is a refinement of the slope stability data $(\overline{\mathbb{Q}}, \{P_{\psi}\}_{\psi\in\overline{\mathbb{Q}}})$.
\end{prop}

The following result describes semistable subcategories for finest stability data on $\coh \mathbb{T}$.

\begin{prop} \label{semistable obj in elliptic curve}
Let $(\Phi,\{\Pi_{\varphi}\}_{\varphi\in\Phi})$ be a finest stability data on $\coh \mathbb{T}$. Then each semistable subcategory has the form $\langle S_{q,{\textbf{\textit{x}}}} \rangle$ for some $ q \in \overline{\mathbb{Q}}, {\textbf{\textit{x}}} \in \mathbb{T}$.
Further, the phases for the stable sheaves satisfy:
\begin{itemize}
  \item[(1)] $ \bm \phi(S_{p,{\textbf{\textit{x}}}}) <  \bm \phi(S_{q,{\textbf{\textit{y}}}});
  \quad
  \forall \ p < q,  {\textbf{\textit{x}}},  {\textbf{\textit{y}}}\in \mathbb{T}; $
  \item[(2)] there are no restrictions on the choice of ordering on the set $\{\bm \phi(S_{p, {\textbf{\textit{x}}}}), \;{\textbf{\textit{x}}} \in \mathbb{T}\}$.
\end{itemize}
 Consequently, any stability data on $\coh \mathbb{T}$ can be refined to a finest one.
\end{prop}

\begin{proof}
By Lemma \ref{semistable objects for T}, we know that any indecomposable sheaf is semistable in $\coh \mathbb{T}$. Notice that for any $p, q \in \overline{\mathbb{Q}}, \; {\textbf{\textit{x}}},  {\textbf{\textit{y}}}\in \mathbb{T}$, $\Hom(S_{p,{\textbf{\textit{x}}}}, $ $ S_{q,{\textbf{\textit{y}}}}) \neq 0$ if and only if $p< q$, or $(p,{\textbf{\textit{x}}})=(q,{\textbf{\textit{y}}})$.
Then the result follows from Theorem \ref{iff cond for finest} and Lemma \ref{semistable objects for T} immediately.
\end{proof}

Note that $\overline{\mathbb{Q}}=\mathbb{Q}\cup \{\infty\}$ is a linearly ordered set in the natural way, i.e., $p<q<\infty$ for any $p<q \in\mathbb{Q}$. Meanwhile, $\mathbb{T}$ can be viewed as a linearly ordered set by fixing any choice of ordering on the points of the elliptic curve $\mathbb{T}$. Hence, for any $q\in\overline{\mathbb{Q}}$, $\{q\}\times\mathbb{T}$ becomes a linearly ordered set by fixing any choice of ordering on the points of $\mathbb{T}$, and then the product $\overline{\mathbb{Q}}\times \mathbb{T}=\bigcup_{q\in\overline{\mathbb{Q}}}(\{q\}\times\mathbb{T})$ also forms a linearly ordered set under the lexicographical order.
There are infinitely many ordered relations on the set $\overline{\mathbb{Q}}\times \mathbb{T}$, depending on the choices of ordering on the set $\{q\}\times\mathbb{T}$ for $q\in\overline{\mathbb{Q}}$.

As a consequence of Proposition \ref{semistable obj in elliptic curve}, we obtain the following classification result for finest stability data on $\coh \mathbb{T}$, parameterized by the linearly ordered set $\overline{\mathbb{Q}}\times \mathbb{T}$.

\begin{prop} \label{Finest stab data in elliptic curve} Any finest stability data on $\coh \mathbb{T}$ has the form $$(\overline{\mathbb{Q}}\times \mathbb{T}, \{\Pi_{(q,{\textbf{\textit{x}}})}\}_{(q,{\textbf{\textit{x}}})\in\overline{\mathbb{Q}}\times \mathbb{T}}),$$ where $\overline{\mathbb{Q}}\times \mathbb{T}$ is an arbitrary linearly ordered set defined as above, and $\Pi_{(q,{\textbf{\textit{x}}})}=\langle S_{q,{\textbf{\textit{x}}}} \rangle$ for any $(q,{\textbf{\textit{x}}})\in\overline{\mathbb{Q}}\times \mathbb{T}$.
\end{prop}

Note that each stability data on $\coh \mathbb{T}$ is coarser than a finest one. Then by Proposition \ref{torsion pair construction},  any torsion pair $(\cal{T}, \cal{F})$ can be obtained from the finest stability data.
As a consequence of Propositions \ref{Finest stab data in elliptic curve} and \ref{torsion pair construction}, we can obtain a classification of torsion pairs in $\coh \mathbb{T}$.

\begin{prop}
The torsion pairs $(\cal{T},\cal{F})$ in $\coh \mathbb{T}$ are classified as below, where $q\in\overline{\mathbb{Q}}$ and $P\subseteq\mathbb{T}$.
\end{prop}

\begin{table}[h]
\caption{All torsion pairs $(\cal{T},\cal{F})$ in $\coh \mathbb{T}$}
\begin{tabular}{|c|c|c|c|c|}
  \hline
 \makecell*[c]{$\cal{T}$}&$\cal{F}$\\
 \hline
 \makecell*[c]{$\langle  S_{p,{\textbf{\textit{x}}}} \mid   p>q, {\textbf{\textit{x}}} \in \mathbb{T}; \;\;\text{or} \;\; p=q, {\textbf{\textit{x}}}\in P\rangle$}&$\langle  S_{p,{\textbf{\textit{x}}}} \mid   p<q, {\textbf{\textit{x}}} \in \mathbb{T}; \;\;\text{or} \;\; p=q, {\textbf{\textit{x}}}\in \mathbb{T}\backslash P\rangle$ \\
 \hline
  \makecell*[c]{$\coh \mathbb{T}$}&$0$ \\
 \hline
 \end{tabular}
 \label{table 3}
\end{table}

\section{Weighted projective lines}

In this section we investigate the stability data for the category $\coh\mathbb{X}$ of coherent sheaves over a weighted projective line $\mathbb{X}$. We show that any stability data on $\coh\mathbb{X}$ can be refined to a finest one when $\mathbb{X}$ is of domestic type or tubular type. Moreover, we classify all the finest stability data on $\coh\mathbb{X}$ for $\mathbb{X}$ having weight type (2), as a by-product we obtain a classification of torsion pairs.

\subsection{\rm{\textbf{Weighted projective lines}}}

Following \cite{GL,Lenzing}, a \emph{weighted projective line} $\X=\X_ \mathbf{k}$
over $\mathbf{k}$ is given by a weight sequence ${\bf
p}=(p_{1},\ldots, p_{t})$ (called \emph{weight type}) of positive integers, and a sequence
${\boldsymbol\lambda}=(\lambda_{1},,\ldots, \lambda_{t}) $ of pairwise
distinct closed points in the projective line
$\mathbb{P}^{1}:=\mathbb{P}^{1}_ \mathbf{k} $ which can be normalized as
$\lambda_{1}=\infty, \lambda_{2}=0, \lambda_{3}=1$. More precisely,
let $\bbL=\bbL(\bf p)$ be the rank one abelian group (called  \emph{string group}) with generators
$\vec{x}_{1}, \ldots, \vec{x}_{t}$ and the relations
\[ p_{1}\vx_1=\cdots=p_{t}\vx_t=:\vec{c},\]
where $\vec{c}$ is called the \emph{canonical element} of
$\mathbb{L}$.
Each element $\vx\in\bbL$ has the \emph{normal form}
$\vx=\sum\limits_{i=1}^t l_i\vx_i+l\vc$ with $0\leq l_i\leq p_i-1$
and $l\in\bbZ$.

Denote by $S$ the commutative algebra
$$S=S({\bf p},{\boldsymbol\lambda})= \mathbf{k} [X_{1},\cdots, X_{t}]/I:
=  \mathbf{k}[x_{1}, \ldots, x_{t}],$$ where $I=(f_{3},\ldots,f_{t})$ is the
ideal generated by
$f_{i}=X_{i}^{p_{i}}-X_{2}^{p_{2}}+\lambda_{i}X_{1}^{p_{1}}$ for
$3\leq i\leq t$. Then $S$ is $\mathbb{L}$-graded by setting
$$\mbox{deg}(x_{i})=\vx_i\; \text{ for $1\leq i\leq t$.}$$
Finally, the weighted projective line associated with $\bf p$ and
$\boldsymbol\lz$ is defined to be
$$\X=\X_ \mathbf{k} ={\rm{Proj}}^{\bbL}S,$$
 the projective spectrum of $\bbL$-graded homogeneous prime ideals of $S$.

Denote by $p={\rm l.c.m.}(p_1,p_2,\cdots,p_t)$. There is a surjective group homomorphism $\delta\colon \bbL\rightarrow \mathbb{Z}$ given by $\delta(\vec{x}_i)=\frac{p}{p_i}$ for $1\leq i\leq t$.
Recall that the \emph{dualizing element} $\vec{\omega}$ in $\bbL$ is defined as $\vec{\omega}=(t-2)\vec{c}-\sum_{i=1}^t \vec{x}_i$. Hence we have $\delta(\vec{\omega})=p((t-2)-\sum_{i=1}^t\frac{1}{p_i})$.
The the weighted projective line $\X$ is called of \emph{domestic, tubular or wild types} provided that $\delta(\vec{\omega})<0$, $\delta(\vec{\omega})=0$ or $\delta(\vec{\omega})>0$ respectively. More precisely, we have the following trichotomy for $\X$ according to the weight sequence $\mathbf{p}$ (up to permutation of weights):
\begin{itemize}
  \item [(i)] domestic type: $(), (p), (p_{1},p_{2})$, $(2,2,p_3)$, $(2,3,3)$, $(2,3,4)$ and $(2,3,5)$;
  \item [(ii)] tubular type: $(2,2,2,2)$, $(3,3,3)$, $(4,4,2)$  and $(6,3,2)$;
  \item [(iii)] wild type: all the other cases.
\end{itemize}
Here, $()$ stands for the unweighted case, that is, the classical projective line.

The category of coherent sheaves on $\X$ is defined to be the
quotient category
$$\coh\X={\rm mod}^{\mathbb{L}}S/\mbox{mod}_{0}^{\mathbb{L}}S,$$
where ${\rm mod}^{\mathbb{L}}S$ is the category of finitely
generated $\mathbb{L}$-graded $S$-modules, while
$\mbox{mod}_{0}^{\mathbb{L}}S$ is the Serre subcategory of finite
length $\mathbb{L}$-graded $S$-modules. For any $\vec{x}\in\bbL$, the grading shift $\vec{x}: E\mapsto E(\vec{x})$ gives an equivalence of $\coh\X$.
Moreover, $\coh\X$ is a hereditary abelian category with Serre
duality of the form
\begin{equation*} \Ext^1(X, Y)\cong\D\Hom(Y, X(\vec{\omega})).\end{equation*}
This implies the existence of
almost split sequences in $\coh\X$ with the Auslander--Reiten
translation $\tau$ given by the grading shift with $\vec{\omega}$.

It is known that $\coh \mathbb{X}$ admits a splitting torsion pair $({\rm{coh}}_0 \mathbb{X}, \vect \mathbb{X})$, where ${\rm{coh}}_0 \mathbb{X}$ and $\vect \mathbb{X}$ are full subcategories of torsion sheaves and vector bundles, respectively.
The subcategory ${\rm{coh}}_0 \mathbb{X}$ splits into a coproduct of connected tube subcategories $\coprod _{{\textbf{\textit{x}}}\in\mathbb{P}^{1}} \textbf{T}_{{\textbf{\textit{x}}}}$,  whose  associated AR-quivers $\Gamma(\textbf{T}_{{\textbf{\textit{x}}}})$ are stable tubes of finite rank. The subcategory $\vect \mathbb{X}$ contains line bundles, which have the forms $\co(\vec{x})$ for $\vec{x}\in\mathbb{L}$.

\subsection{\rm{\textbf{Finest stability data on $\coh  {\mathbb{X}}$ of domestic or tubular type}}}

In this subsection, we investigate the stability data on the category $\coh  {\mathbb{X}}$ of coherent sheaves for weighted projective lines of domestic or tubular type. We will show that each stability data on $\coh\mathbb{X}$ can be refined to a finest one.

Recall that there is a slope function $\mu: \ind (\coh  {\mathbb{X}})\to \mathbb{Q}\cup \{\infty\}$ on the category of coherent sheaves over $\mathbb{X}$, satisfying $\mu(E)=\deg(E) / \rank(E)$ for any sheaf $E$. Then $\mu(E)\in \mathbb{Q}$ when $E$ is a vector bundle, and $\mu(X)=\infty$ for any torsion sheaf $X$.
Denote by $\overline{\mathbb{Q}}=\mathbb{Q}\cup \{\infty\}$, which is a  linearly ordered set in the natural way.
This slope function gives a stability data $(\overline{\mathbb{Q}}, \{P_{\psi}\}_{\psi\in\overline{\mathbb{Q}}})$, called \emph{slope stability data}, where $$\overline{\mathbb{Q}}=\mathbb{Q}\cup \{\infty\}, \quad P_{\psi}=\langle X\in \ind (\coh \mathbb{X})\;|\; \mu(X)=\psi\rangle.$$
By assumption, i.e., $\X$ is of domestic or tubular type, each indecomposable sheaf is semistable. Moreover, $\Hom(X, Y)\neq 0$ implies $\mu(X)\leq \mu(Y)$ for any indecomposable sheaves $X,Y$. This slope function gives a decomposition of the category $\cal{A}=\coh \mathbb{X}=\add\{\cal{A}_{q}\,|\,q\in\overline{\mathbb{Q}}\}$, where $\cal{A}_{q}=\langle X\in \ind (\coh \mathbb{X})\;|\; \mu(X)=q\rangle$.

Recall that each subcategory $\cal{A}_q$ has one of the following forms:
\begin{itemize}
   \item [(1)] $\cal{A}_q$ is a semisimple category generated by finitely many orthogonal exceptional vector bundles, in case $\mathbb{X}$ is domestic type and $q\in \mathbb{Q}$;
   \item [(2)] $\cal{A}_q$ is an abelian category consisting of a family of connected orthogonal tube subcategories parameterized by the projective line, in case $q=\infty$ or $\mathbb{X}$ is of tubular type.
\end{itemize}

The main result of this subsection is as follows.

\begin{prop}\label{finest for wpl abelian} Each stability data on $\coh\mathbb{X}$ can be refined to a finest one.
\end{prop}

\begin{proof}
Let $(\Phi,\{\Pi_{\varphi}\}_{\varphi\in\Phi})$ be a stability data on $\cal{A}:=\coh\mathbb{X}$. We first claim that
$(\Phi,\{\Pi_{\varphi}\}_{\varphi\in\Phi})$ can be refined to a finer one with
each semistable category is contained in some $\cal{A}_{q}:=\langle X\in \ind (\coh \mathbb{X})\;|\; \mu(X)=q\rangle$.

In fact, for any $\varphi\in\Phi$, define $I_{\varphi}=\{(\varphi,q)\ |\ q\in \overline{\mathbb{Q}},\,\, \Pi_{\varphi} \cap \cal{A}_q\neq 0\}$ and set $\Psi=\cup_{\varphi\in\Phi} I_{\varphi}$. Then  $\Psi$ is a linearly ordered set under the lexicographical order, i.e.,
 $(\varphi', q')>(\varphi'', q'')$ if and only if $\varphi'>\varphi''$, or $\varphi'=\varphi''$ and $q'>q''$.
Set $P_{\varphi, q}=\Pi_{\varphi} \cap \cal{A}_q$ for any $(\varphi,q)\in\Psi$.
Then by Proposition \ref{local refinement construction}, we obtain that
$(\Psi,\{P_{\varphi,q}\}_{(\varphi,q)\in\Psi})$ is a finer stability data than $(\Phi,\{\Pi_{\varphi}\}_{\varphi\in\Phi})$, with each $P_{\varphi,q}$ contained in some $\cal{A}_q$.

Recall that for weighted projective line of domestic or tubular type, the subcategory $\cal{A}_q$ is generated by finitely many orthogonal exceptional vector bundles, or $\cal{A}_q$ is a coproduct of connected orthogonal tube subcategories. Hence we can further make local refinement on $(\Phi,\{\Pi_{\varphi}\}_{\varphi\in\Phi})$ such that
each semistable category $\Pi_\varphi$ is generated by a unique exceptional vector bundle, or contained in a homogeneous tube, or contained in a non-homogeneous tube.

For the first two cases, the subcategory $\Pi_\varphi$ satisfies $\Hom(X,Y)\neq 0\neq \Hom(Y,X)$
for non-zero objects $X,Y\in\Pi_{\varphi}$.
For the third case, the semistable subcategory $\Pi_{\varphi}$ is a subcategory of a tube category. Then by similar arguments as in the proof of Proposition \ref{finest stab data for cyclic quiver},
we can make local refinement to obtain a finest stability data. This finishes the proof.
\end{proof}

\subsection{\rm{\textbf{Semistable sheaves on $\coh  {\mathbb{X}}$ for weight type (2)}}}

From now on, we focus on the weighted projective line $\mathbb{X}$ of weight type (2).
Recall that the group  $\mathbb{L} = \mathbb{L}(2)$ is the rank one abelian group with generators
$\vec{x}_{1}, \vec{x}_{2}$ and the relations $$ 2\vec{x}_{1}=\vec{x}_{2}:=\vec{c}.$$
Each element $\vx\in\bbL$ has the \emph{normal form}
$\vx=l\vc$ or $\vx=\vec{x}_1+l\vc$ for some $l\in\bbZ$. In this case, each indecomposable bundle in $\vect \mathbb{X}$ is a line
bundle, hence has the form $\mathcal{O}(\vec{x}), \vec{x} \in \mathbb{L}$.

The subcategory ${\rm{coh}}_0 \mathbb{X}(2)$ admits ordinary simple sheaves $S_{{\textbf{\textit{x}}}}$ for each ${\textbf{\textit{x}}} \in \mathbb{P}^{1} \setminus \{\infty\}$
and exceptional simple sheaves $S_{1,0}, S_{1,1}$.
The ordinary simple sheaf $S_{{\textbf{\textit{x}}}}$ is determined by the exact sequence
$$ 0\rightarrow \mathcal{O} \xrightarrow{X_{2}-{\textbf{\textit{x}}}X_{1}^{2}} \mathcal{O}(\vec{c})\longrightarrow  S_{{\textbf{\textit{x}}}} \rightarrow 0.$$
By contrast, multiplication by $X_{1}$ leads to the exceptional simples $S_{1,  j}$ for $ j\in\mathbb{Z}/2 \mathbb{Z}$:
$$ 0\rightarrow \mathcal{O}((j-1)\vec{x}_{1})\stackrel{X_{1}} \longrightarrow \mathcal{O}(j\vec{x}_{1})\rightarrow  S_{1, j} \rightarrow 0.$$

From now on, we assume $(\Phi,\{\Pi_{\varphi}\}_{\varphi\in\Phi})$ is a finest stability data on $\coh \mathbb{X}(2)$.
Obviously, all the ordinary simple sheaves $S_{{\textbf{\textit{x}}}}$ and exceptional simples $S_{1,  0}, S_{1,  1}$ in $\coh \mathbb{X}(2)$ are stable since they have no non-trivial subsheaves. Note that $S_{1,1}(\vec{\omega})=S_{1,0}$. Up to degree shift, we can assume $$\bm{\phi}(S_{1,0})<\bm{\phi}(S_{1,1}).$$
Then $S^{(2n)}_{1,1}$ is also semistable for any $n\geq 1$, and $\bm{\phi}(S_{1,0})<\bm{\phi}(S^{(2)}_{1,1})=
\bm{\phi}(S^{(2n)}_{1,1})<\bm{\phi}(S_{1,1})$ since $\langle S_{1,1}^{(2)}\rangle=\add\{S_{1,1}^{(2n)}, n\geq 1\}$.

For any line bundle $L$, we will denote by $S_{i, L} $ the unique exceptional simple sheaf satisfying  that $\Hom (L, S_{i, L} )\neq 0$.
For examples, $S_{1,  \mathcal{O}}=S_{1,  0}, S_{1,  \mathcal{O}(\vec{x}_{1})}=S_{1,  1} $.

\begin{lem} \label{semistable line bundle a)}
If a line bunle $L$ is not semistable, then $L(-\vec{x}_{1})$ is semistable and the HN-filtration of $L$ is given by the exact sequence
$$0\rightarrow L(-\vec{x}_{1}) \rightarrow  L\rightarrow S_{1,L}\rightarrow0.$$
\end{lem}

\begin{proof}
Consider the  HN-filtration of $L$
$$\xymatrix {
&\cdots L_{n-2}\ar@{^(->}[r]^-{}&L_{n-1}\ar@{->>}[d]^{}\ar@{^(->}[r]^-{}&
L_{n}=L,\ar@{->>}[d]^{}\\
&&A_{n-1}&A_n
 }
$$
then the epimorphism $L \twoheadrightarrow A_n$ implies $A_n$ is a torsion sheaf. For any $i\leq n-1$, $\Hom(A_{i}, A_n)=0$ implies $\Hom(L_{n-1}, A_n)=0$, so we have $A_n=S_{1,0} $ or $S_{1,1}$, and then $L_{n-1}= L(-\vec{x}_{1})$. If $L_{n-1}$ is not semistable, then $L_{n-2}\not=0$ and $A_{n-1}=S_{1,1} $ or $S_{1,0}$ by similar arguments as above. Note that $A_{n-1}\neq A_{n}$. Hence $\{A_{n-1}, A_{n}\}=\{S_{1,0}, S_{1,1}\}$. So $\Hom(L_{n-2}, S_{1,0})=0=\Hom(L_{n-2}, S_{1,1})$, a contradiction. Hence $L_{n-1}$ is  semistable.
We are done.
\end{proof}

\begin{lem} \label{semistable line bundle b)}
If $L$ and $L(\vec{x}_{1})$ are both semistable line bundles, then
$L(-\vec{x})$ are semistable for all $\vec{x}\geq 0$.
\end{lem}

\begin{proof}
Obviously, it suffices to show that $L(-\vec{x}_{1})$ is semistable. Then the result follows by induction. For contradiction we assume $L(-\vec{x}_{1})$ is not semistable. Then by Lemma \ref{semistable line bundle a)}, the HN-filtration of $L(-\vec{x}_{1})$ is given by: $0\rightarrow L(-\vec{c})\rightarrow L(-\vec{x}_{1})\rightarrow S_{1,L(-\vec{x}_{1})}\rightarrow0$. But $\Hom(L(\vec{x}_{1}), S_{1,L(-\vec{x}_{1})})\not=0$ implies $\bm \phi(L(\vec{x}_{1})) <\bm \phi(S_{1,L(-\vec{x}_{1})}) <\bm \phi(L(-\vec{c}))$, a contradiction to $\Hom(L(-\vec{c}), L(\vec{x}_{1}))\not=0$. We are done.
\end{proof}

For any $m\in\mathbb{Z}$, we denote by
$$\mathbb{L}_{m}:= \{ k\vec{c} \ | \ k\in \mathbb{Z}, k< m  \} \cup \{\vec{x}_{1}+\mathbb{Z}\vec{c} \}$$ for the subset of $\mathbb{L}$.
Then semistable line bundles with respect to a finest stability data $(\Phi,\{\Pi_{\varphi}\}_{\varphi\in\Phi})$ can be described as follows.

\begin{prop}  \label{semistable line bundle in weight type (2)}
For any finest stability data $(\Phi,\{\Pi_{\varphi}\}_{\varphi\in\Phi})$ on $\coh \mathbb{X}(2)$, let $\{\mathcal{O}(\vec{x})\ | \ \vec{x} \in X\}$ be the set of all the semistable line bundles. Then $X=\vec{x}_{1}+\mathbb{Z}\vec{c}, \;\mathbb{L},$ or $\mathbb{L}_{m}$ for some $m\in\mathbb{Z}$.
\end{prop}

\begin{proof}
Recall that we have assumed $\bm{\phi}(S_{1,0})<\bm{\phi}(S_{1,1}).$ We claim that all $\mathcal{O}(\vec{x}_{1}+k\vec{c})$ are semistable for $k\in \mathbb{Z}$.
In fact, assume there exists some $\mathcal{O}(\vec{x}_{1}+k\vec{c})$ which is not semistable, then by Lemma \ref{semistable line bundle a)}, the HN-filtration of $\mathcal{O}(\vec{x}_{1}+k\vec{c})$ is given by $$0\rightarrow \mathcal{O}(k\vec{c}) \rightarrow  \mathcal{O}(\vec{x}_{1}+k\vec{c}) \rightarrow S_{1,1}\rightarrow0.$$
Hence $\bm \phi(\mathcal{O}(k\vec{c}))> \bm \phi(S_{1,1})>\bm \phi(S_{1,0})$. Then  $\Hom(\mathcal{O}(k\vec{c}), S_{1,0})\neq 0$ yields a contradiction. This proves the claim.
It follows that $\vec{x}_{1}+\mathbb{Z}\vec{c}\,\subseteq\, X\subseteq\,\mathbb{L}$.

Assume $X\neq \vec{x}_{1}+\mathbb{Z}\vec{c}$ and $X\neq\mathbb{L}$, then we need to show that $X=\mathbb{L}_{m}$ for some $m\in\mathbb{Z}$.
Note that $\mathbb{L}=\mathbb{Z}\vec{c}\cup\{\vec{x}_{1}+\mathbb{Z}\vec{c}\}$.
Since $X\neq \vec{x}_{1}+\mathbb{Z}\vec{c}$, there exists some $k\in\mathbb{Z}$, such that $\mathcal{O}(k\vec{c})$ is semistable. Since $\mathcal{O}(k\vec{c}+\vec{x}_1)$ is semistable, by Lemma \ref{semistable line bundle b)} we obtain that $\mathcal{O}(\vec{x})$ is semistable for any $\vec{x} \leq k\vec{c}$. By assumption $X\neq\mathbb{L}$, then there exists a minimal $m$, such that $\mathcal{O}(m\vec{c})$ is not semistable. It follows that $\mathcal{O}((m-1)\vec{c})$ is semistable and $\mathcal{O}(k\vec{c})$ is not semistable for any $k\geq m$. That is, $X=\mathbb{L}_{m}$. We are done.
\end{proof}

\subsection{\rm{\textbf{Finest stability data and torsion pairs on $\coh  {\mathbb{X}(2)}$}}}

Now we are in the position to give a classification of finest stability data on $\coh\mathbb{X}(2)$. For this let's introduce some linearly ordered sets as below.

Note that $\mathbb{L}$ is a linearly ordered set in the natural way, i.e., $\vec{x}\leq\vec{y}$ if and only if $\vec{y}-\vec{x} \geq 0$ in $\mathbb{L}$. Then the subsets $\mathbb{L}_{m} \, (m\in\mathbb{Z})$ and $\vec{x}_{1}+\mathbb{Z}\vec{c}$ inherit the linearly ordered relations of $\mathbb{L}$ naturally.
Set  $$\widetilde{\mathbb{X}}=\{ (\infty,  0), (\infty,  \frac{1}{2}), (\infty,  1),  {\textbf{\textit{x}}}  \ | \ {\textbf{\textit{x}}}\in \mathbb{P}^{1} \setminus \{\infty\} \}. $$ Then $\widetilde{\mathbb{X}}$ can be viewed as a linearly ordered set by fixing any choice of ordering on the elements of $\widetilde{\mathbb{X}}$ satisfying $(\infty,  0)< (\infty,  \frac{1}{2})< (\infty,  1)$.

By taking union of two ordered sets and composing some new relations between them, we can obtain the following three family of linearly ordered sets, depending on the choices of ordering on the set $\widetilde{\mathbb{X}}$:

\begin{itemize}
  \item  $\mathbb{L}\cup \widetilde{\mathbb{X}}$: by composing $\vec{x}<{\textbf{\textit{z}}}$ for any $\vec{x} \in\mathbb{L}$ and ${\textbf{\textit{z}}}\in\widetilde{\mathbb{X}}$.

\item  $\{\vec{x}_{1}+\mathbb{Z}\vec{c}\}\cup \widetilde{\mathbb{X}}$: by composing $(\infty, 0) < \vec{y}<{\textbf{\textit{z}}}$ for any $\vec{y}
\in \vec{x}_{1}+\mathbb{Z}\vec{c}$ and ${\textbf{\textit{z}}}\in\widetilde{\mathbb{X}}\setminus \{ (\infty,  0) \}$.

\item  $\mathbb{L}_{m}\cup \widetilde{\mathbb{X}}$: by composing $(m-1)\vec{c}< (\infty,  0) < \vec{x}_{1}+(m-1)\vec{c}$, and $\vec{y}<{\textbf{\textit{z}}}$  for any $\vec{y} \in\mathbb{L}_{m}$ and ${\textbf{\textit{z}}}\in\widetilde{\mathbb{X}}\setminus \{ (\infty,  0) \}$.
\end{itemize}

The main result of this subsection is as follows, which gives a classification of finest stability data of $\coh \mathbb{X}(2)$ under the assumption $\bm{\phi}(S_{1,0})<\bm{\phi}(S_{1,1})$.

\begin{prop} \label{Finest stab data in weight type (2)} Keep notations as above. For any $\varphi\in\mathbb{L}\cup\widetilde{\mathbb{X}}$, define the subcategories $\Pi_{\varphi}\subseteq\coh  {\mathbb{X}(2)}$ as below:
\begin{itemize}
  \item  $\Pi_{\vec{x}}= \langle \mathcal{O}(\vec{x}) \rangle, \vec{x} \in \mathbb{L} $;
  \item  $\Pi_{(\infty,  0)}= \langle S_{1,  0} \rangle $; \quad $\Pi_{(\infty,  \frac{1}{2})}= \langle S_{1,  1}^{(2)} \rangle $;\quad  $\Pi_{(\infty,  1)}= \langle S_{1,  1} \rangle $;
   \item  $\Pi_{{\textbf{\textit{x}}}}= \langle S_{{\textbf{\textit{x}}}} \rangle,  {\textbf{\textit{x}}}\in \mathbb{P}^{1} \setminus \{\infty\} $.
\end{itemize}
Then there are only three types of the finest stability data $(\Phi,\{\Pi_{\varphi}\}_{\varphi\in\Phi})$ on $\coh \mathbb{X}(2)$, where $\Phi=\mathbb{L}\cup \widetilde{\mathbb{X}}$, $\mathbb{L}_{m}\cup \widetilde{\mathbb{X}}$ or $\{\vec{x}_{1}+\mathbb{Z}\vec{c}\}\cup \widetilde{\mathbb{X}}$ respectively.
\end{prop}

\begin{proof}
 We first prove that $(\Phi,\{\Pi_{\varphi}\}_{\varphi\in\Phi})$ is a finest stability data for $\Phi=\{\vec{x}_{1}+\mathbb{Z}\vec{c}\}\cup \widetilde{\mathbb{X}}$. Indeed, it is easy to see that $\Hom(\Pi_{\varphi'}, \Pi_{\varphi''})=0$ for any $\varphi'>\varphi''\in\Phi$. Note that $\langle S_{1,j}^{(2)}\rangle=\add\{S_{1,j}^{(2n)}, n\geq 1\}$ for $j=0,1$. Hence all the indecomposable sheaves which do not belong to any $\Pi_{\varphi}$ are given by
\begin{itemize}
    \item [-] $\co(k\vec{c}),\; k\in\mathbb{Z}$;
    \item [-] $S^{(2n+1)}_{1,1}, \;n\geq 1;$ \quad $S^{(2n+1)}_{1,0}, \;n\geq 1$; \quad $S^{(2n+2)}_{1,0}, \;n\geq 0$.
\end{itemize}
They admit HN-filtrations of the following forms:
\begin{equation*}
{\footnotesize\xymatrix { \co(\vec{x}_{1}+(k-1)\vec{c})\ar@{^(->}[r]^-{} &\co(k\vec{c})\ar@{->>}[d]^{} && &&S_{1,1}^{}\ar@{^(->}[r]^-{}&S^{(2n+1)}_{1,1}\ar@{->>}[d]^{}&&\\
&S_{1,0}&&&& &S^{(2n)}_{1,1} &&}}
\end{equation*}

\begin{equation*}
\xymatrix {
&&&S_{1,1}^{(2n)} \ar@{^(->}[r]^-{}&S^{(2n+1)}_{1,0}\ar@{->>}[d]^{}&&& S_{1,1} \ar@{^(->}[r]^-{}&S_{1,1}^{(2n+1)}\ar@{->>}[d]^{}\ar@{^(->}[r]^-{}&S^{(2n+2)}_{1,0}.
\ar@{->>}[d]^{} &&\\
&&&&S^{}_{1,0} &&&&S_{1,1}^{(2n)}&S^{}_{1,0}&&\\
}
\end{equation*}
Hence $(\Phi,\{\Pi_{\varphi}\}_{\varphi\in\Phi})$ is a stability data on $\coh \mathbb{X}(2)$. Observe that each semistable subcategory is generated by an indecomposable sheaf. It is easy to see that $\Hom(X,Y)\neq 0\neq \Hom(Y,X)$ for any $\varphi\in\Phi$ and any non-zero $X,Y\in\Pi_{\varphi}$. According to Theorem \ref{iff cond for finest}, the stability data $(\Phi,\{\Pi_{\varphi}\}_{\varphi\in\Phi})$ is finest.

Similarly, one can prove that $(\Phi,\{\Pi_{\varphi}\}_{\varphi\in\Phi})$ is a finest stability data on $\coh \mathbb{X}(2)$ for $\Phi=\mathbb{L}\cup \widetilde{\mathbb{X}}$ or $\mathbb{L}_{m}\cup \widetilde{\mathbb{X}}$.

On the other hand, for any finest stability data $(\Phi,\{\Pi_{\varphi}\}_{\varphi\in\Phi})$ on $\coh \mathbb{X}(2)$, we know that the simple sheaves are stable since they don't have non-trivial subsheaves. Up to degree shift, we can assume $\bm \phi(S_{1,  0})<\bm \phi(S_{1,  1})$. Hence $S_{1,1}^{(2)}$ is semistable and $S_{1,0}^{(2)}$ is not semistable. It follows that $S_{1,1}^{(2n)}$ is semistable with $\bm{\phi}(S_{1,0})<\bm{\phi}(S^{(2)}_{1,1})=\bm{\phi}(S^{(2n)}_{1,1})<\bm{\phi}(S_{1,1})$ and $S_{1,0}^{(2n)}$ is not semistable for any $n\geq 1$. According to Lemma \ref{non semistable objects}, we know that
$S_{1,j}^{(2n+1)}$ is not semistable for any $n\geq 1$ and $j=0,1$. Therefore, the semistable sheaves in the subcategory ${\rm{coh}}_0\mathbb{X}(2)$ are characterized by the linearly ordered set $\widetilde{\mathbb{X}}$.
Moreover, according to Proposition \ref{semistable line bundle in weight type (2)}, the semistable line bundles are precisely characterized by one of the three linearly ordered sets $\vec{x}_{1}+\mathbb{Z}\vec{c}$, $\mathbb{L}$ or
$\mathbb{L}_{m}\,(m\in\mathbb{Z})$.
Note that ${\rm{Hom}}(\Pi _{\varphi'},\Pi_{\varphi''})=0$ for all $\varphi'>\varphi''$ in $\Phi$. Then the finest stability data on $\coh \mathbb{X}(2)$ are exhausted by the above three types.

This finishes the proof.
\end{proof}

\begin{rem} For the classical projective line or an arbitrary smooth elliptic curve, any finest stability data on its category of coherent sheaves is a refinement of the slope stability data by Propositions \ref{refine from slope for P1} and \ref{refine from slope  for T}, respectively. But this is not the case for weighted projective lines. In fact, the finest stability data $(\Phi,\{\Pi_{\varphi}\}_{\varphi\in\Phi})$ on $\coh \mathbb{X}(2)$ for $\Phi=\{\vec{x}_{1}+\mathbb{Z}\vec{c}\}\cup \widetilde{\mathbb{X}}$ is not finer than the slope stability data, since $\bm{\phi}(S_{1,0})=(\infty, 0) <\vec{x}_{1}+k\vec{c}=\bm{\phi}(\mathcal{O}(\vec{x}_{1}+k\vec{c}))$ for any $k\in\mathbb{Z}$, but $\mu(S_{1,0})>\mu(\mathcal{O}(\vec{x}_{1}+k\vec{c}))$.
\end{rem}

As an application of Proposition \ref{Finest stab data in weight type (2)}, we can obtain a classification of torsion pairs in $\coh \mathbb{X}(2)$.

\begin{prop} \label{torsion pairs on wpl (2)}
Up to degree shift, the non-trivial torsion pairs $(\cal{T}, \cal{F})$  in $\coh \mathbb{X}(2)$  are classified as below, where $P$ is a non-empty subset of $\mathbb{P}^{1}$ and $Q$ is a (possibly empty) subset of $\mathbb{P}^{1} \setminus \{\infty\}$.
\end{prop}

\begin{table}[h]
\caption{Non-trivial torsion pairs $(\cal{T},\cal{F})$ in $\coh \mathbb{X}(2)$}
\begin{tabular}{|c|c|c|c|c|}
  \hline
 \makecell*[c]{$(\cal{T}, \cal{F})$}&$\cal{T}$&$\cal{F}$\\
\hline
 \makecell*[c]{$I$}&$\langle S_{{\textbf{\textit{x}}}} \mid  {\textbf{\textit{x}}}\in P  \rangle$&$  \vect \mathbb{X}(2) \cup \langle S_{{\textbf{\textit{x}}}} \mid  {\textbf{\textit{x}}}\not\in P \rangle  $ \\
  \hline
 \makecell*[c]{$II$}&$\langle S_{1,1}, S_{{\textbf{\textit{x}}}} \mid  {\textbf{\textit{x}}}\in Q \rangle$&$ \vect \mathbb{X}(2) \cup \langle S_{1,0}, S^{(2)}_{1,1},  S_{{\textbf{\textit{x}}}} \mid  {\textbf{\textit{x}}}\not\in Q \rangle   $\\
  \hline
 \makecell*[c]{$III$}&$\langle S^{(2)}_{1,1}, S_{1,1}, S_{{\textbf{\textit{x}}}} \mid  {\textbf{\textit{x}}}\in Q\rangle$&$ \vect \mathbb{X}(2) \cup \langle  S_{1,0}, S_{{\textbf{\textit{x}}}} \mid  {\textbf{\textit{x}}}\not\in Q \rangle  $ \\
  \hline
 \makecell*[c]{$IV$}&$\langle \mathcal{O}(\vec{y}) \mid  \vec{y}\geq \vec{0} \rangle  \cup {\rm{coh}}_0 \mathbb{X}(2)$&$\langle \mathcal{O}(\vec{y}) \mid \vec{y}< \vec{0}\rangle$\\
  \hline
 \makecell*[c]{$V$}&$\langle \mathcal{O}(\vec{x}_{1}+ t\vec{c}),  S^{(2)}_{1,1}, S_{1,1}, S_{{\textbf{\textit{x}}}} \mid   t\geq 0, {\textbf{\textit{x}}}\in \mathbb{P}^{1}  \rangle$&$\langle  \mathcal{O}( \vec{x}), S^{}_{1,0} \mid  \vec{x} < \vec{x}_{1} \rangle   $ \\
  \hline
 \makecell*[c]{$VI$}&$\langle \mathcal{O}( \vec{x}_{1}+t\vec{c}),  S^{(2)}_{1,1}, S_{1,1}, S_{{\textbf{\textit{x}}}} \mid  t\in \mathbb{Z}, {\textbf{\textit{x}}}\in \mathbb{P}^{1} \rangle $&$\langle  S_{1,0}\rangle$\\
  \hline

 \end{tabular}
 \label{table 4}
\end{table}

\begin{proof}
For any finest stability data $(\Phi, \{\Pi_{\varphi}\}_{\varphi\in\Phi})$ on $\coh \mathbb{X}(2)$, the simple sheaves $S_{1,0}$ and $S_{1,1}$ are stable. Up to degree shift, we can assume $$\bm{\phi}(S_{1,0})<\bm{\phi}(S_{1,1}).$$ According to Corollary \ref{torsion pair determined by phi}, the following subcategories
$$\Pi_{\geq\psi}=\langle\Pi_{\varphi}\mid\varphi\geq \psi \rangle,\quad\Pi_{<\psi}=\langle\Pi_{\varphi}\mid\varphi<\psi\rangle;$$
$$\Pi_{>\psi}=\langle\Pi_{\varphi}\mid\varphi> \psi \rangle,\quad\Pi_{\leq\psi}=\langle\Pi_{\varphi}\mid\varphi\leq\psi\rangle$$
give two torsion pairs $(\Pi_{\geq\psi}, \Pi_{<\psi})$ and $(\Pi_{>\psi}, \Pi_{\leq\psi})$ in $\coh \mathbb{X}(2)$. These torsion pairs, besides the trivial ones, i.e., $(0, \coh \mathbb{X}(2))$ and $(\coh \mathbb{X}(2), 0)$, have the forms in Table \ref{table 4} due to Proposition \ref{Finest stab data in weight type (2)} (up to degree shift), where $P$ is a non-empty subset of $\mathbb{P}^{1}$, and $Q$ is a (possibly empty) subset of $\mathbb{P}^{1} \setminus \{\infty\}$.

On the other hand, by Proposition \ref{finest for wpl abelian}, each stability data on $\coh \mathbb{X}(2)$ can be refined to a finest one. Then by Proposition \ref{torsion pair construction}, the classification in Table 6 is complete.
\end{proof}

\begin{rem} The torsion pairs in the category $\coh \mathbb{X}$ of coherent sheaves for domestic and tubular weighted projective lines $\mathbb{X}$ have been described in \cite{CS}, via certain bounded t-structures $(\mathcal{D}'^{\leq 0}, \mathcal{D}'^{\geq 0})$ with $\mathcal{D}^{\leq -1} \subset \mathcal{D}'^{\leq 0} \subset \mathcal{D}^{\leq 0}$ on the bounded derived category $\mathcal{D}^{b}(\coh \mathbb{X})$, where $(\mathcal{D}^{\leq 0}, \mathcal{D}^{\geq 0})$ is the standard bounded t-structure on $\mathcal{D}^{b}(\coh \mathbb{X})$ with heart $\coh \mathbb{X}$.
More precisely, for weighted projective line of weight type $(2)$, there are two different types of non-trivial torsion pairs $(\cal{T}, \cal{F})$:

(1) $(\cal{T}, \cal{F})$ is a tilting torsion pair, corresponding to $(\mathcal{D}'^{\leq 0}, \mathcal{D}'^{\geq 0})$ with finite length heart;

(2) $\cal{T} \subset {\rm{coh}}_0 \mathbb{X}(2)$ or $\cal{F} \subset {\rm{coh}}_0 \mathbb{X}(2)$, corresponding to $(\mathcal{D}'^{\leq 0}, \mathcal{D}'^{\geq 0})$ not having finite length heart.

In fact, up to degree shift, the torsion pairs of the second type correspond to $I, II, III$ or $VI$, while the tilting torsion pairs correspond to $IV$ or $V$ in Table \ref{table 4}. Moreover, the tilting torsion pair IV is induced by the canonical tilting bundle $\mathcal{O}\oplus\mathcal{O}(\vec{x}_1)\oplus\mathcal{O}(\vec{c})$, and the tilting torsion pair V is induced by the tilting sheaf $\mathcal{O}(\vec{x}_{1})\oplus\mathcal{O}(\vec{x}_{1}+ \vec{c})\oplus S_{1,1}$.
\end{rem}

\bigskip

{\bf Acknowledgments.}
This work was supported by the National Natural Science Foundation of China (Nos. 12271448, 11971398), the Natural Science Foundation of Fujian Province (No. 2022J01034), Natural Science Foundation of Xiamen (No. 3502Z20227184) and the Fundamental Research Funds for Central Universities of
China (No. 20720220043).

\end{document}